\documentclass[a4paper,11pt]{amsart}
\RequirePackage{filecontents}        % loading package filecontents
\usepackage[utf8]{inputenc}

\usepackage{geometry} 
\usepackage{authblk}
\usepackage[foot]{amsaddr}

\usepackage{graphicx} 

\usepackage{float}			
\usepackage{comment}
\usepackage{graphicx} 
\usepackage{caption}
\usepackage{subcaption}
\usepackage{verbatim} 
\usepackage[shortlabels]{enumitem}

\usepackage{amsmath}

\usepackage{amssymb}
\usepackage{lmodern} 
\usepackage{mathtools}
\usepackage{bm}
\usepackage{esvect}
\usepackage[colorlinks]{hyperref}
\usepackage{bbm}
\usepackage{dsfont}
\usepackage[T1]{fontenc}
\usepackage{amsfonts}
\usepackage[font=small,labelfont=bf]{caption}

\usepackage{amsthm}
\numberwithin{equation}{section}

\theoremstyle{plain}
\newtheorem{thm}{Theorem}[section]

\newtheorem{lem}[thm]{Lemma}
\newtheorem*{lem*}{Lemma}
\newtheorem{prop}[thm]{Proposition}

\newtheorem{cor}[thm]{Corollary}

\newtheorem{asp}[thm]{Assumption}
\newtheorem{Condition}[thm]{Condition}

\newtheorem{defn}[thm]{Definition}

\newtheorem{rem}[thm]{Remark}
\newtheorem*{rem*}{Remark}

\makeatletter
\newcommand*\bigcdot{\mathpalette\bigcdot@{.5}}
\newcommand*\bigcdot@[2]{\mathbin{\vcenter{\hbox{\scalebox{#2}{$\m@th#1\bullet$}}}}}
\makeatother

\newcommand\NN{\mathbb{N}} % Ensemble des entiers naturels.
\newcommand\ZZ{\mathbb{Z}} % Ensemble des entiers relatifs.
\newcommand\RR{\mathbb{R}} % Ensemble des nombres réels.
\newcommand\R{\mathbb{R}} % Ensemble des nombres réels.
\newcommand\un{\mathbbm{1}}
\newcommand\EE{\mathbb{E}} 
 
\newcommand\PP{\mathbb{P}}

\newcommand\restr[2]{{
  \left.\kern-\nulldelimiterspace 
  #1 
  \vphantom{\big|} % pretend it's a little taller at normal size
  \right|_{#2} % this is the delimiter
  }}

\newcommand\compl{\mathsf{c}}

\DeclareMathOperator{\cov}{cov}

\DeclareMathOperator{\Ber}{Ber}

\DeclareMathOperator{\Var}{Var}
\DeclareMathOperator{\Vol}{Vol}

\DeclareMathOperator{\Diam}{Diam}
\DeclareMathOperator{\Supp}{Supp}

\begin{document}
\title[Variance bounds for Gaussian FPP]{Variance bounds for Gaussian first passage percolation}
\author{Vivek Dewan$^*$}
\email{vivek.dewan@univ-grenoble-alpes.fr}
\address{$^*$Institut Fourier, Universit\'{e} Grenoble Alpes}
%\date
\begin{abstract}
Recently, many results have been established drawing a parallel between Bernoulli percolation and models given by levels of smooth Gaussian fields with unbounded, strongly decaying correlation (see e.g \cite{beffara2016v}, \cite{rivera2017critical}, \cite{muirhead2018sharp}).
In a previous work with D. Gayet \cite{dewangayet}, we started to extend these analogies by adapting the first basic results of classical first passage percolation (first established in \cite{kesten1986aspects}, \cite{cox}) in this new framework: positivity of the time constant and the ball-shape theorem. In the present paper, we present a proof inspired by Kesten \cite{10.1214/aoap/1177005426} of other basic properties of the new FPP model: an upper bound on the variance in the FPP pseudometric given by the Euclidean distance with a logarithmic factor, and a constant lower bound.
Our results notably apply to the Bargmann-Fock field.
\end{abstract}
\date{\today}
\thanks{}

\maketitle
\tableofcontents
\section{Introduction}
\subsection{The models}
\medskip

\paragraph{\bf{Classical FPP}} The classical model of first passage percolation (FPP) was introduced by Hammersley and Welsh in 1965~\cite{Hw}. In its most basic form, it consists in assigning an i.i.d $\Ber(p)$ random variable (seen as a \emph{time}) to every edge in the graph $(\ZZ^d,\mathcal{E})$, where the edges in $\mathcal{E}$ are all edges between pairs of vertices which differ by $\pm1$ in one coordinate. The pseudometric $T$ is then defined as the smallest number of $1$ weights in an edge path between two vertices. One important quantity in the study of this model is the family of \emph{time constants}: the deterministic limits of $T(0,nx)/n$ for any given $x$,  denoted $\mu_p(x)$.
It is known that there is for this quantity a phase transition similar to that of Bernoulli percolation: $\mu_p$ is positive if and only if $p$ is smaller than $p_c(d)$, which is the critical parameter for Bernoulli percolation, and depends on the dimension $d$. Subsequently, Kesten \cite{10.1214/aoap/1177005426} established results controlling the fluctuations of the quantity $T(0,nx)$, namely that its variance was lower bounded by a constant and linearly upper-bounded. In later works, Benjamini, Kalai and Schramm \cite{BKS}, (and Benaïm and Rossignol \cite{10.1214/07-AIHP124} with relaxed conditions on the law) got a logarithmic improvement on the upper bound and Newman and Piza \cite{newmanpiza} got one on the lower bound. 
\medskip

\paragraph{\bf{Gaussian FPP}}
Recall that a Gaussian field is a random function $\RR^d\mapsto \RR$ such that for any finite set of points $(x_1,...,x_k)$, $(f(x_1),...,f(x_k))$ is a Gaussian vector. A Gaussian field is fully determined by its \emph{covariance kernel}
$$
\kappa(x,y):=\cov(f(x),f(y)).
$$
In recent years, there has been increased interest in a percolation model based on such random maps: \emph{Gaussian percolation}. It is  \emph{a priori} widely different from classical percolation. It pertains to the large-scale behaviour of excursion sets of smooth Gaussian fields, i.e sets of the form
$$
\mathcal{E}_\ell :=\{x\text{ }|\text{ } f(x)\geq -\ell\}.
$$
A phase transition similar to that of Bernoulli percolation was established in the planar case, with the parameter $\ell$ of the threshold level playing the role of $p$ (the first properties in \cite{beffara2016v}, the full phase transition for the Bargmann-Fock field in \cite{rivera2017critical}, and finally for planar fields with polynomial decay in \cite{muirhead2018sharp}). 
Higher dimensional phase transition has been studied in recent papers: a sharp phase transition has been established first for fields with bounded correlation \cite{dewan2021upper}, then for fields with fast enough polynomial decorrelation \cite{severo2021sharp}.
The \emph{critical level} $\ell_c$ (a priori depending on the field $f$) is defined as 
\begin{equation}\label{ellcdef}
\ell_c:=\sup\{\ell \text{ such that }\PP[\mathcal{E}_\ell\text{ has an unbounded conected component}]=0\}.
\end{equation}
In a recent paper \cite{dewangayet}, D. Gayet and the author have started to investigate a FPP model in the context of Gaussian fields, with a pseudometric naturally defined from excursion sets. We established a time constant result: in this model and under several natural assumptions on the field, the time constant $\mu (x)$ is positive if and only if the level considerd $\ell$ is positive, i.e "most of the space" has full time cost. In the present paper, we will establish both upper and lower bounds on the variance of $T(0,x)$, in the framework of Gaussian fields with exponential decay of correlations, with ideas directly adaptated from Kesten's \cite{10.1214/aoap/1177005426}.

\subsection{Previous results}
We present a general defintion for our pseudometric.
\begin{defn}\label{pseudometric}
Let $\psi$ be a measurable function $\RR\mapsto \RR$ such that
\begin{enumerate}
\item $\psi\geq 0$
\item $\psi$ is non-decreasing
\item \label{sublinear} There exists a constant $C_\psi >0$ such that for any $x\geq 1$, $\psi(x)\leq C_\psi x$.
\item $\psi(x)>0 \Leftrightarrow x>0.$
\end{enumerate}
Let $f$ be an almost-surely continuous Gaussian field over $\RR^d$, $A$ and $B$ be two compact subsets of $\RR^d$ and $\ell\in\RR$. The associated pseudometric is then:
\begin{equation}
T(A,B):=\inf\limits_{\substack{\gamma\text{ piecewise affine}\\ \text{path from $A$ to $B$}}} \quad \int\limits_\gamma \psi(f+\ell).
\end{equation}
\end{defn}

\begin{rem}
The following two values of $\psi$ yield very natural pseudometrics. 
\begin{itemize}
\item The first one can be seen as a Gaussian equivalent of classical FPP with Bernoulli edge weights:
\begin{equation*} \psi=\un_{\RR_+^*}.
\end{equation*}
\item The second one can be pictured as a metric given by the graph of the Gaussian field, with a "flat sea" leveling off the low values:
\begin{equation*}
 \psi(x)=\max(x,0).
\end{equation*}
\end{itemize}
\end{rem}
For technical reasons, given any pair of sets $A,B$ both contained in a third set $\mathcal{S}$, we define the \emph{restricted pseudometric} as
\begin{equation}\label{restrtime}
T^\mathcal{S}(A,B):=\inf\limits_{\substack{\gamma\text{ piecewise affine path}\\ \text{from $A$ to $B$ contained in $\mathcal{S}$}}} \quad \int\limits_\gamma \psi(f+\ell),
\end{equation}
Notice that for any pair of Gaussian fields $f,g$ such that $f\leq g$, for any bounded sets $A,B$, for any $\ell$, $T(A,B)$ is larger when evaluated with respect to $g$ than when evaluated with respect to $f$.
For any $x\in\RR^d$, we call the \emph{time constant} associated with $x$ the real number $\mu (x)$ such that
	\begin{equation}\label{convergences}
 \lim\limits_{n \to +\infty} \frac{1}n T(0,nx) = \mu (x)\ \text{  almost surely and $L^1$},
	\end{equation}
provided it exists.

The main result of the paper \cite{dewangayet} concerning Gaussian FPP was the following (see next section for statement of the assumptions).

\begin{thm}(\cite[Theorems 2.5 and 2.7]{dewangayet})\label{posBF}
Let $f$ be a Gaussian field over $\RR^d$ and satisfying Assumption \ref{a:basic} for some $\alpha$-sub-exponential function $F$ with $\alpha>1$. Let $\ell\in\RR$.
Let $T$ an associated pseudometric given by~(\ref{pseudometric}). Then, 
\begin{enumerate}
	\item the associated family of time constants $(\mu )_{\ell\in \RR} $ given by (\ref{convergences}) are well defined, and they are either all zero or all non-zero (but finite), in which case $\mu $ is a norm.
	 \item If $B_t$ is the ball of radius $t$ for the pseudometric $T$, and $\mathcal{B}_M$ is the ball of radius $M$ for the sup norm,\begin{enumerate}
		\item If $\mu =0$

		then for any positive $M$, 
		$$ \PP  \Bigl[
				\mathcal{B}_M \subset \frac{1}t B_t
				\text{ for all $t$ large enough}
				\Bigr] =1,$$
where $\mathcal{B}_M$ is the Euclidean ball of radius $M$.
		\item If $\mu $ is a norm then there exists a convex  compact subset $K$ of $\RR^d$ with non-empty interior such that, for any positive $\epsilon$,
		\begin{equation}\label{inclusions}
		\PP  \Bigl[(1-\epsilon)K\subset\frac{1}tB_t\subset(1+\epsilon)K\text{ for all $t$ large enough}\Bigr]=1.
		\end{equation}

	\end{enumerate}
	\item  \label{postimecst} Assume further that $f$ satisfies the further Assumption \ref{a:pos} (positivity of correlations). Then, 
$$
 \ell>-\ell_c\Rightarrow  \mu  >0 $$
	\item \label{postimecst2}  Assume further that $f$ is planar. Then, 
$$
 \mu  >0  \Leftrightarrow  \ell>0.$$
	
\end{enumerate}
\end{thm}
\begin{rem}\label{higher}
\begin{enumerate}
\item Items (\ref{postimecst}) and (\ref{postimecst2}) can be intuitively interpreted as the idea that what matters in knowing whether the time constant is positive is whether the instantaneous set percolates: the condition $\ell>-\ell_c$ exactly means that the instantaneous set is below the percolation threshold (item (\ref{postimecst}) is only an implication but it is expected to be an equivalence).
\item
	In the original work, these results were established in a more general framework than that of Gaussian fields, which only made use of assumptions of decorrelation and decay of one-arm probabilities. Notably, they were also valid for an FPP model given by random Voronoi tilings.
\end{enumerate}
	\end{rem} 
This result as well as our new bounds (Theorems \ref{main2} and \ref{main}) apply to the \emph{Bargmann-Fock field}, which appears in random complex and real algebraic geometry (see \cite{beffara2016v}). It is given by the correlation kenel:
$$\kappa(x,y)=\exp\left(-\frac{1}{2}\|x-y\|^2\right).
$$
Equivalently, we can explicitly write it as the following random field $f$:
\begin{equation}\label{BF}
f(x)=\exp\left(-\frac{1}{2}|x|^2\right)\sum\limits_{i,j\in\NN} a_{i,j}\frac{x_1^ix_2^j}{\sqrt{i!j!}},
\end{equation}
where the $a_{i,j}$'s are i.i.d centered Gaussians of variance $1$. 
\bigskip

\paragraph{\bf{Acknowledgements}}
The author is grateful to Damien Gayet for his many corrections and enlightening discussions. We also thank Stephen Muirhead for valuable insights and comments on an earlier version of this work. We also warmly thank the referee who carefully read and helpfully commented on an earlier version of this paper.
\section{Main results}
\subsection{Definitions and Assumptions}
We first define the \emph{finitely correlated} counterparts of a given Gaussian field, which we will be using repeatedly in what follows.
\begin{defn}\label{corrrange}
We say that a Gaussian field $f$ over $\RR^d$ has \emph{correlation range $r$} for some $r>0$ if its covariance kernel $\kappa$ verifies $\kappa(x,y)=0$ for all $x,y$ such that $\|x-y\|_2 \geq r$.
\end{defn}

\begin{defn}\label{d:cutoff}
Fix some smooth function $\varphi$ on $\RR$ such that
\begin{itemize}
\item $\varphi\geq 0,$
\item $\Supp(\varphi)\subseteq [-1,1]$,
\item $\varphi=1$ on $[-1/2,1/2]$.
\end{itemize}
We then define, for any Gaussian field $f=q\star W$ and $r>0$, the counterpart of $f$ with correlation range $r$ to be:
$$
f_r:=q_r\star W,
$$
where 
$$
q_r(x)=\varphi(\|x\|_2/r)q.
$$
\end{defn}
\paragraph{\bf{Notations}}\label{notations}
In the rest of the paper, the function $\psi$ (see Definition \ref{pseudometric}) giving the pseudometric is considered to be fixed.
Notations $\PP$, $\EE$, $\Var$, $\cov$ will denote probability, expectation, variance, covariance respectively. Any random variable or event containing the pseudometric $T$ is to be interpreted in the sense of Definition \ref{pseudometric}, the field $f$ and the level $\ell$ always being fixed beforehand. 
For any $r>0$, the index $r$ signifies that we switch to the law where we take the field $f_r$ instead of $f$. For example, if $E$ is an event, we write 
$$
\PP_r[E]
$$
to signify the probability of $E$ for the field $f_r$. 
When there are more than one Gaussian fields considered, we may also put the name of the field as the index (e.g write $\PP_f$ for the probability of an event for the field $f$). If $E$ is a set or an event, notations $\un_E$ and $\un\{E\}$ both designate the indicator function of $E$.

Now some definitions relating to correlation decay.
\begin{defn}\label{subpoly}
A function $F$ defined on $\RR$ is said to be $\alpha$\emph{-sub-polynomial} for some $\alpha>0$ if,
$$
x^\alpha F(x) \longrightarrow 0.
$$
A function $F$ defined on $\RR$ is said to be $\alpha$\emph{-sub-exponential} for some $\alpha>0$ if,
$$
e^{x^\alpha} F(x) \longrightarrow 0.
$$
\end{defn}
\begin{defn}
For any number $r$, we let $A_r$ be the annulus of inner radius $1$ and outer radius $r$ centered at $0$ for the sup norm. We call $S_r$ the boundary of the ball of radius $r$, and we write $T(A_r)$ for $T(S_1,S_r)$.
\end{defn}
The following is the key condition used in all of our results. While it does not seem very natural in its statement, it can be seen as a quantitative estimate in the positivity of the time constant.
\begin{Condition}\label{macrotime}
A Gaussian field on $\RR^d$ along with a pseudometric of the form given by Definition \ref{pseudometric} satisfies the \emph{macroscopic annulus times condition} for some level $\ell\in\RR$ if there exists a constant $a>0$ and a $\max(2d,4)$-sub-polynomial function $G$ such that for any  $N\geq 1$, 
$$
\PP (T(A_N)\leq aN)\leq G(N).
$$
\end{Condition}
As we will see thanks to Proposition \ref{numberballs}, in the cases where we have positivity of the time constant, as established in \cite{dewangayet}, Condition \ref{macrotime} is always verified.
Here are the main assumptions on the Gaussian field $f$ we will use for all of our results. 
\begin{asp}(Basic assumptions)
\label{a:basic}
\begin{enumerate}[(a)]
\item \label{i:stat} The field $f$ is centered, stationary and ergodic.
\item \label{i:whitenoise} The field $f$ has a \emph{spatial-moving-average representation} $f=q\star W$ where $q\in L^2(\RR^d)$ and $W$ is the white-noise on $\RR^d$.
\item(Regularity) \label{i:reg} $q$ is $\mathcal{C}^3$ and each of its derivatives is in $L^2(\RR^d)$. Further, $q$ is $L^1$.
\item(Decay of correlations with map $F$) \label{i:decay} There exists a function $F$ from $\RR$ to $\RR$ which decays to $0$ at $\infty$ such that for any $x\in\RR$, for any multi-index $\alpha$ such that $|\alpha|\leq 1$, $$\Bigg(\int\limits_{\|u\|\geq x}|\partial^\alpha q(u)|^2\Bigg)^{1/2}\leq F(x).$$
\item(Symmetry)\label{i:sym} The moving-average kernel $q$ is symmetric under permutation of the axes, and symmetry in all axes.
\item(Positive spectral density)\label{i:spectr} The moving-average kernel $q$ verifies $\int\limits_{\RR^d} q>0$.
\end{enumerate}
\end{asp}
\begin{rem}
Assumption \ref{i:whitenoise} is implied by, the covariance kernel $\kappa$ having fast enough decay. When it holds, then we have the equality $q\star q=\kappa(0,.)$, and further the Fourier tranform of $q$ is the square root of that of $\kappa(0,.)$, and is a continuous function. For a definition of the white-noise, see section \ref{p:WN} of the appendix. Assumption \ref{i:reg} ascertains that $f$ is almost surely $\mathcal{C}^2$.
\end{rem}
Further, we will use the following extra assumption for some of our results:
\begin{asp}(Positive association)
\label{a:pos}
We say that a Gaussian field $f$ verifies the \emph{positive association assumption} if its moving-average kernel $q$ verifies
$$
q\geq 0.
$$
\end{asp}
\begin{rem}
Recalling that $q\star q=\kappa(0,.)$, this assumption implies
$$
\kappa\geq 0.
$$
\end{rem}

In the rest of the article, the norm used on $\RR^d$ is the sup norm, and for any $r$, $\mathcal{B}_r$ denotes the ball of radius $r$ for said norm. Within proofs, all numbered constants depend only on the field $f$ and level $\ell$ considered, and two constants with the same number in two different proofs may be different.

\subsection{Statements}
Our first result is the following constant lower bound.
\begin{thm}\label{main2}
Let $f$ be a Gaussian field on $\RR^d$ verifying Assumption \ref{a:basic} as well as the positivity assumption, Assumption \ref{a:pos}. Let $\ell\in\RR$ be a level. Let $T$ be a pseudometric as defined in Definition \ref{pseudometric}.  Then there exists a constant $C>0$  such that for any $|x|\geq 2$, 
$$
\Var  T(0,x)\geq C.
$$
\end{thm}
Our main general result, an upper bound on the variance, is the following:
\begin{thm}\label{main}
Let $f$ be a Gaussian field on $\RR^d$ verifying Assumption \ref{a:basic} with $\alpha$-sub-exponential decay function $F$ for some $\alpha>1$. Let $\ell\in\RR$ be a level. Let $T$ be a pseudometric as defined in Definition \ref{pseudometric}. Suppose that Condition \ref{macrotime} is verified. Fix $\varepsilon>0$. Then there exists a constant $C_\varepsilon>0$  such that for any $|x|\geq 2$, 
$$
\Var  T(0,x)\leq C_\varepsilon|x|(\log|x|)^{1/\alpha+\varepsilon}.
$$
\end{thm}
In particular, using Proposition \ref{numberballs}, we have
\begin{cor}\label{2dcoro}
Let $f$ be a Gaussian field on $\RR^d$ verifying Assumption \ref{a:basic} with $\alpha$-sub-exponential decay function $F$ for some $\alpha>1$ and the positivity assumption, Assumption \ref{a:pos}. Let the level verify $\ell>-\ell_c(f)$ (the critical level of the field $f$ as defined in (\ref{ellcdef})). Let $T$ be a pseudometric as defined in Definition \ref{pseudometric}. Fix $\varepsilon>0$. Then there exists a constant $C_\varepsilon>0$ such that for any $|x|\geq 2$, 
$$
\Var  T(0,x)\leq C_\varepsilon|x|(\log|x|)^{1/\alpha+\varepsilon}.
$$

\end{cor}
%Our lower bound is the following:
%\begin{thm}\label{main2}
%Consider any Gaussian field on $\RR^2$ verifying Assumption \ref{a:basic} with $\alpha$-sub-exponential decay function $F$ for some $\alpha>1$ and the positivity assumption \ref{a:pos}.  Consider a level $\ell>0$. Let $T$ be a pseudometric as defined in Definition \ref{pseudometric}. Then for any $\varepsilon>0$ there exists a constant $C_\varepsilon>0$ such that for any $|x|\geq 2$, 
%$$
%\Var  T(0,x)\geq C_\varepsilon(\log|x|)^{1-2/\alpha-\varepsilon},
%$$
%where $T$ is a pseudometric as defined in Definition \ref{pseudometric}.
%\end{thm}
%\begin{rem}\label{lowerbdrmk}
%This lower bound is only interesting when we have $\alpha>2$, which is an extremely fast decay: even the Bargmann-Fock field does not qualify. 
%\end{rem}
\paragraph{\bf{Open questions}}
\begin{itemize}
%\item As Remark \ref{lowerbdrmk} points to, it would be interesting to improve the lower bound to one where the exponent in the logarithm is positive for a wider class of Gaussian fields.
\item Though we strongly suspect it to be the case, it remains to ascertain whether our methods can be adapted to other models in which we can find ways to split randomness into disjoint boxes, as we will be doing here. Notably, for the Poisson-Boolean or Voronoi FPP models (see e.g \cite{dewangayet} for their definitions).
\item It would be interesting to relax the assumptions on the Gaussian field to a slower than exponential decay. The sharpness of phase transition result from \cite{severo2021sharp} which we use to establish Condition \ref{macrotime} only requires a polynoial decay with exponent larger than the dimension, and we expect that our result should hold in this regime as well.
\item One might wonder why our upper bound doesn’t match the best bound in
the classical framework obtained by Benjamini, Kalai and Schramm
in \cite{BKS}. Our main tool, the Efron-Stein inequality (Proposition \ref{ES})
is the one used by Kesten to obtain a linear bound, and we do not
rely on methods based on Poincaré inequalities as in \cite{BKS} and \cite{10.1214/07-AIHP124}. Such
inequalities have been established for Gaussian measures. However,
to our knowledge, these apply in a framework where one has a control of order $\|x\|$ on the
the number of real random variables which intervene in computing
the functional $T(0,x)$, which is not the case in our framework since the
underlying space is a Gaussian white-noise.
\end{itemize}
\section{Proof of main results}
\subsection{Proof of the constant lower bound}
We give a proof of our constant lower bound, Theorem \ref{main2}, which does not need Condition \ref{macrotime} but does need the positive correlation assumption, Assumption \ref{a:pos}. It is strongly inspired by Kesten's original FPP proof in \cite{10.1214/aoap/1177005426} and its restatement by Auffinger, Damron and Hanson in \cite{auffinger201750}.
Before the proof itself, we only need one auxiliary result: the following decomposition proved by S. Muirhead and the author in a previous work \cite{dewan2021upper}.
\begin{defn}\label{efess}
Let $f$ be a Gaussian field satisfying Assumption \ref{a:basic} with moving-average kernel $q$ and let $r>0$. For any $r>0$, we define the field $\tilde{f}_r$ to be the field 
$$ q \star (W|_{\mathcal{B}_r}),
$$
see section \ref{p:WN} of the appendix for details on restrictions of the white-noise.
\end{defn}
The decomposition is as follows.
\begin{prop}[\cite{dewan2021upper}, Proposition A.1]\label{decomp}
Let $f$ be a Gaussian field satisfying Assumption \ref{a:basic} with moving-average kernel $q$ and let $r>0$.
Let $Z_1$ be a standard normal random variable. Then there exists a Gaussian field $g$ independent from $Z_1$ such that we have the following equality in law:
\[ \tilde{f}_r(\cdot) \stackrel{d}{=} \frac{Z_1 (q \star \un_{\mathcal{B}_r})(\cdot)}{r^{d/2}} + g(\cdot) .\]
\end{prop}
On to the proof itself.
\begin{proof}[Proof of Theorem \ref{main2}]
Let $f$ be a Gaussian field verifying the assumptions, and $\ell\in\RR$. By Proposition \ref{decomp}, we have the equality in law:
\begin{equation}\label{equality}
 \tilde{f}_{1} \stackrel{d}{=} Z_1 (q \star \un_{\mathcal{B}_{1}}) + g ,
\end{equation}
where we recall (Definition \ref{efess}) that $\tilde{f}_1=q\star (W|_{\mathcal{B}_1})$, $Z_1$ is a standard Gaussian and $g$ an independent Gaussian field.
Let  $Z_1'$ be an independent standard Gaussian.  We write $f$ (resp. $f'$) for the associated realization of the Gaussian field using $Z_1$ (resp. $Z_1'$), and a common realization of $g$ and $q\star (W|_{\mathcal{B}_1^\compl})$ (see Proposition \ref{splitboxes} for a justification of this splitting of the white-noise).
Let $A$ be the event $\{Z_1\geq 1\}$ and $B$ be the event $\{Z'_1\leq -1\}$, let $a>0$ be their common probability.
We now write, for all $|x|\geq 2$, 
\begin{align}
\begin{split}
\label{variancelowerbd}
&\Var_f T(0,x) \geq \Var\EE_f\Big[ T(0,x)\Big| Z_1\Big]\\
&\quad =\frac12  \EE\Bigg[\Big(\EE_f\Big[ T(0,x)\Big| Z_1\Big]-\EE_{f'}\Big[ T(0,x)\Big| Z_1'\Big]\Big)^2\Bigg]\\
&\quad \geq \frac12  \EE\Bigg[\Big(\EE_f\Big[ T(0,x)\Big| Z_1\Big]-\EE_{f'}\Big[ T(0,x)\Big| Z_1'\Big]\Big)^2\un_{A\cap B}\Bigg].
\end{split}
\end{align}
For any $\varepsilon>0$, define the constant $C_0$, independent of $x$, to be:
$$
C_0:=\EE\Bigg[\inf\limits_{\substack{\gamma\text{ piecewise affine}\\ \text{ path from $0$ to }\partial\mathcal{B}_\varepsilon }}T_f(\gamma)-T_{f'}(\gamma)\Bigg| A\cap B\Bigg],
$$
where $T_f(\gamma)$ (resp. $T_{f'}(\gamma)$) denotes the integral along $\gamma$ of $\psi(f+\ell)$ (resp. $\psi(f'+\ell)$). Now, $q\geq 0$  by Assumption \ref{a:pos} (and $q$ is non-zero by Assumption \ref{a:basic}), and on $A\cap B$, $Z_1-Z_1'\geq 2$. Plugging this into (\ref{equality}), we deduce that there exists a constant $C_f>0$ such that on event $A\cap B$, $f\geq f'+C_f$ on $\mathcal{B}_1$. For any $\varepsilon>0$, consider the event
$$
E_\varepsilon:=\{ C_f/4-\ell\leq f\leq C_f/2-\ell\text{ on $\mathcal{B}_\varepsilon$}\}.
$$
Notice that $E_\varepsilon\cap A$ and $B$  are independent. Further, event $E_\varepsilon$ can be written conditionally on the value of $Z_1$ as an event of the form $g|_{\mathcal{B}_\varepsilon}\in I$, where $g$ is the independent field from (\ref{equality}) and $I$ is some interval of nonempty interior. So that $E_\varepsilon\cap A$ has positive probability for $\varepsilon$ small enough. We fix such an $\varepsilon$.
In total, the event $E_\varepsilon\cap A\cap B$ has positive probability.
We then have, by Definition \ref{pseudometric} that on $E_\varepsilon\cap A\cap B$,
$$
\inf\limits_{\substack{\gamma\text{ piecewise affine}\\ \text{ path from $0$ to }\partial\mathcal{B}_\varepsilon }}T_f(\gamma)>0,
$$
and 
$$
f'|_{\mathcal{B}_\varepsilon}\leq -\ell \text{, hence  }T_{f'}|_{\mathcal{B}_\varepsilon}=0.
$$
Then, since
$$
C_0\geq \EE\Bigg[\Big(\inf\limits_{\substack{\gamma\text{ piecewise affine}\\ \text{path from $0$ to }\partial\mathcal{B}_\varepsilon }}T_f(\gamma)-T_{f'}(\gamma)\Big)\un_{E_\varepsilon}\Bigg| A\cap B\Bigg],
$$
we conclude that 
$$
C_0>0.
$$
Finally, we claim that when event $A\cap B$ occurs, 
$$
\EE_f\Big[ T(0,x)\Big| Z_1\Big]-\EE_{f'}\Big[ T(0,x)\Big| Z_1'\Big]\geq C_0.
$$
Indeed, if $\gamma$ is a path between $0$ and $x$, we have by definition of $C_0$ that on $A\cap B$,
$$\EE_f T(\gamma|_{\mathcal{B}_\varepsilon}) \geq \EE_{f'}T(\gamma|_{\mathcal{B}_\varepsilon})+C_0.$$
Further, since $T$ is an increasing random variable (see Definition \ref{increasing}) and $q\geq 0$, $Z_1\geq Z_1'$ implies that  $T_f\geq T_{f'}$. So that, on $A\cap B$,
$$T_f(\gamma|_{\mathcal{B}_\varepsilon^\compl}) \geq T_{f'}(\gamma|_{\mathcal{B}_\varepsilon^\compl}).$$
In total, on $A\cap B$,
$$\EE_f T(\gamma) \geq \EE_{f'}T(\gamma)+C_0.$$
And, returning to (\ref{variancelowerbd}), for any $|x|\geq 2$,
$$
\Var T(0,x) \geq \frac12(C_0a^2)^2.
$$
\end{proof}
\subsection{Proof of the upper bound}
We start with the auxiliary results used in the proof of our main upper bound. The longer proofs are delayed to section \ref{proofaux}. 
\subsubsection{Auxiliary results}
The first lemma gives us a control of the sup norm of a Gaussian field in a box of given size, using only its expected pointwise variances and those of its first derivatives.

\begin{prop}[\cite{muirhead2018sharp}, Lemma 3.12]\label{muirvansupnorm}
There exists a constant $c_0>0$ such that for any $\mathcal{C}^1$ Gaussian field $g$ on $\RR^2$, for any $R_1\geq c_0$ and $R_2\geq \log R_1$, and for any $r\in [1,\infty]$
$$
\PP\Big[\|g\|_{\infty,\mathcal{B}_{R_1}}\geq mR_2\Big]\leq e^{-R_2^2/c_0},
$$
where
$$
m=\Big(\sup\limits_{x\in\RR^2}\sup\limits_{|\alpha|\leq 1}\EE[(\partial^\alpha g)^2(x)]\Big)^{1/2}.
$$
\end{prop}
Integrating this estimate yields the following corollary:
\begin{cor}\label{muirvansupnorm2}
For any integer $k$ there exists a constant $C>0$ such that for any $\mathcal{C}^1$ Gaussian field $g$ on $\RR^2$, for any $R_1\geq C$, and for any $r\in[1,\infty]$
$$
\EE\Big[\|g_r\|^k_{\infty,\mathcal{B}_{R_1}}\Big]\leq Cm^k \log^k R_1.
$$
\end{cor}
\begin{rem}
In the case of the fields we will be working with, i.e those satisfying Assumption \ref{a:basic}, we have $m<\infty$.
\end{rem}
One key notion in using the previous estimates is that of monotonic events.
\begin{defn}\label{increasing}
\begin{itemize}
\item A measurable event $A$ of a Gaussian field $f$ over $\RR^d$ is said to be \emph{increasing} if for any non-negative function $h$ on $\RR^d$, $f\in A \implies f+h\in A$. 
\item It is said to be \emph{decreasing} if the same holds for any non-positive function $h$. 
\item Finally, an event is said to be \emph{monotonic} if it is either increasing or decreasing.
\item Similarly, a real-valued random variable $X$ which is a function of a Gaussian field on $\RR^d$ is said to be \emph{increasing} (resp. \emph{decreasing}, \emph{monotonic}) if for any Gaussian fields $f$ and non-negative (resp. non-positive) function $h$ , $X$ evaluated with respect to $f$ is smaller than $X$ evaluated with respect to $f+h$.
\end{itemize}
\end{defn}
Muirhead and Vanneuville established a comparison result for probabilities of monotonic events between a Gaussian field and its finite correlation range versions (see Proposition \ref{CMMV}). It relies on the notion of \emph{Cameron-Martin space} of a Gaussian field (introduced and discussed in section \ref{p:CM} of the appendix). 
We have slightly modified their proof to obtain the following comparison between variances for Gaussian fields with infinite correlation range and their finite-range counterparts.   
\begin{prop}\label{variancesfields}
Let $f$ be a Gaussian field satisying Assumption \ref{a:basic}, with some $\alpha$-sub-exponential decay function $F$ for $\alpha>1$, and $\ell\in\RR$ such that Condition \ref{macrotime} is verified. Let $\varepsilon>0$. Then for any positive integer $N$ and $x$ of norm $N$
$$
\Var T(0,x) \leq \Var_{(\log N)^{1/\alpha+\varepsilon}}[T^{\mathcal{B}_{N^2}}(0,x)](1+o(1))+o(1),
$$
where $o(1)$ designates quantities which go to $0$ as $N$ goes to infinity, and $T^{\mathcal{B}_{N^2}}$ is as in (\ref{restrtime}).
\end{prop}

\medskip

Now for the tools used in establishing an upper bound for the variance in the bounded correlation model.
The first one is a classical technique for estimating the variance of a function of several random variables by "splitting the variance contributed by each one". We will be using it in the context of mesoscopic squares which geodesics for our pseudometric pass through.
\begin{prop}[Efron-Stein's inequality]\label{ES}
Let $(X_1,...,X_n)$ be a finite sequence of independent random variables, and $X'_i$ be an independent copy of $X_i$ for all $i$. Let $\phi$ be an $L^2$ function of $(X_1,...,X_n)$. Then,
$$
\Var \phi(X_1,...,X_n)\leq \sum\limits_{i=1}^{n}\EE\left[\Big(\phi(X_1,..,X_{i-1},X'_i,X_{i+1},...,X_n)- \phi(X_1,,...,X_n)\Big)_+^2\right],
$$
where index $+$ denotes the positive part.
\end{prop}

The following lemma will be used in order to control the number of "small increments" in a geodesic. It is akin to a classical lemma by Kesten (\cite{kesten1986aspects}, Proposition 5.8).
\begin{defn}
For any pair of number $r$, $R$ and any finite sequence of compact sets in $\RR^d$, we say that they are $(r,R)$\emph{-separated} if the distance of any pair of sets in the sequence is more than $r$ and the distance between any consecutive sets is less than $R$.
\end{defn}

\begin{defn}\label{E2}
For any $a,B,M>0$, for any sequence $(r_N)_{N\in\NN}$ of real numbers, for any $N\in\NN$, let  $\mathcal{E}_{a,B,r,N,M}$ be the event that there exists a sequence of more than $M$ disjoint translates of $A_{r_N}$ with centers in $\ZZ^d$, $(r_N,Br_N)$-separated, with the first annulus being at distance less than $Br_N$ from $0$, and with the sum of their times smaller than $aN$.
\end{defn}
\begin{lem}\label{length}
Let $f$ be a Gaussian field on $\RR^d$ field verifying Assumption \ref{a:basic}, as well as $\ell\in\RR$ such that Condition \ref{macrotime} is verified. Let $B>1$. Let $(r_N)_{N\in\NN}$ be a sequence going to $\infty$ as $N$ goes to $\infty$. 
Then there exist $a>0$, $N_0\in\NN$ such that
$$
\forall N\geq N_0\text{, }\forall M\geq 2\frac{N}{r_N}\text{, }\quad\PP_{r_N}[\mathcal{E}_{a,B,r,N,M}]\leq \left(\frac12\right)^{M}.
$$
\end{lem}
\begin{proof}
Fix a Gaussian field, a parameter $\ell$. Consider a sequence of annuli as in Definition \ref{E2} and $a>0$.  Call $(\mathcal{A}_1,...,\mathcal{A}_M)$ the first $M$ annuli of the sequence.
For each annulus $\mathcal{A}_i$, define
$$
\mathcal{V}_{i,a,N}:=\{T(\mathcal{A}_i)\leq ar_N\}.
$$
Further, notice that for $\mathcal{E}_{a,B,r,N,M}$ to occur, at least $M-N/r_N$ events of the form $\mathcal{V}_{i,a,N}$ occur.
The separation condition in $\mathcal{E}_{a,B,r,N,M}$ allows to notice that, for any annulus in the sequence, the next one can be chosen among $C_0(2B)^dr_N^d$ possibilities, $C_0$ being a universal constant. The same holds for the first annulus.
Hence the total number of different possible sequences is smaller than
$$
(C_0(2B)^dr_N^d)^M.
$$
Thus the number of possible sets of annuli with time smaller than $ar_N$ is smaller than
\begin{align*}
(C_0(2B)^dr_N^d)^M{M \choose M-\frac{N}{r_N}}
 \leq (C_0(2B)^dr_N^d)^M(2e)^{M-\frac{N}{r_N}}
\end{align*}
(by the classical inequality ${n \choose k}\leq \left(\frac{en}{k}\right)^k$, since $M\geq 2\frac{N}{r_N}$).
Finally, using the separation assumption in Definition \ref{E2}, we have for all $N$, $M\geq 2\frac{N}{r_N}$,
$$
\PP_{r_N}[\mathcal{E}_{a,B,r,N,M}]\leq (C_0(2B)^dr_N^d)^M \left(2e\PP_{r_N}(\mathcal{V}_{i,a,N})\right)^{M-\frac{N}{r_N}}.
$$
So that, since $M\geq 2\frac{N}{r_N}$,
\begin{equation}\label{repr}
\PP_{r_N}[\mathcal{E}_{a,B,r,N,M}]\leq \left(C_0(2B)^dr_N^d\sqrt{2e\PP_{r_N}(\mathcal{V}_{i,a,N})}\right)^M.
\end{equation}
Since we have Condition \ref{macrotime}, we get that there exists $a>0$, $N_0\in\NN$ such that for all $N\geq N_0$, for all $i$,
$$
\PP_{r_N}(\mathcal{V}_{i,a,N})\leq \frac{1}{8eC_0^2(2B)^{2d}r_N^{2d}}.
$$
The conclusion then follows from (\ref{repr}).
\end{proof}
We now make a few geometric observations which will allow us to apply the previous lemma.
\begin{defn}\label{crossed}
We say that a continuous path $\gamma$ \emph{crosses} an annulus $A$ if $\gamma$ intersects both the inner and outer squares of $A$.
\end{defn}
\begin{defn}\label{properlyused}
Consider any continuous path $\gamma$, and a $d$-dimensional square $\mathcal{S}$ of side $r$.
We call the square $\mathcal{S}$ \emph{properly used} if its intersection with $\gamma$ has diameter greater than $\frac{r}{3^d}$. 
\end{defn} 
\begin{defn}\label{squareslattice}
For any $r>0$ and integer $d$, let $\mathcal{S}$ be the $d$-dimensional hyper-square defined by
$$
\mathcal{S}:=\Big\{x=(x_1,...,x_d)\in\RR^d\big|\quad \forall i\in\{1,...,d\}, \quad 0< x_i<r\Big\}.
$$
We will call \emph{set of hyper-squares defined by the lattice $r\ZZ^d$} the set of translates of $\mathcal{S}$ by a vector of the form
$$
r(k_1,...,k_d),
$$
the $k_i$'s being integers.
\end{defn} 
Let $r>0$ and $f$ be a Gaussian field. Let $\ell>0$. Let $T$ be the metric associated with $f$ and $\ell$ (see Definition \ref{pseudometric}). 
Let $(S_i)_i$ be the set of hyper-squares defined by $r \ZZ^d$. 
\begin{defn}\label{Gox}
For any $x\in\RR^d$, call $\mathcal{G}(0,x)$ the set of squares of $r\ZZ^d$ which are at distance less than $r$ from all geodesics between $0$ and $x$.
\end{defn}
\begin{lem}\label{proper}
Consider a square $S_i$ a path $\gamma$ goes through. Then, as long as $\gamma$ does not \emph{properly use} (see Definition \ref{properlyused}) a neighboring square, it stays within $S_i$ and its neighbors.
\end{lem}
The proof of this lemma is immediate:
\begin{proof}
Let $(\tilde{S}_j)_{j=1,...,3^d-1}$ be the family of neighboring squares of $S_i$, i.e squares that share at least one boundary vertex with $S_i$ in $r\ZZ^d$. If the path $\gamma$ does not \emph{properly use} any of the $\tilde{S}_j$, by adding up the diameter of its intersection with each of them, we deduce that $\gamma$ cannot go at distance larger than $\frac{3^d-1}{3^d}r<r$ from $S_i$, and thus it cannot intersect any square other than $S_i$ or the $\tilde{S}_j$'s.
\end{proof}

For any $x\in\RR^d$, given any geodesic $\gamma$ for the pseudometric $T$ (see Definition \ref{pseudometric}) between $0$ and $x$, we define a $(r,2r)$-separated set of squares which we call $\tilde{\mathcal{G}}_\gamma(0,x)$, by the following procedure:
\begin{defn}\label{algo1}
Let $x\in\RR^d$. Start with $\tilde{\mathcal{G}}_\gamma(0,x)$ being the empty set.
Each new square the geodesic intersects is added to the set $\tilde{\mathcal{G}}_\gamma(0,x)$ if:
\begin{itemize}
\item none of its $3^d-1$ neighbors is already in $\tilde{\mathcal{G}}_\gamma(0,x)$.
\item it is \emph{properly used}.
\end{itemize}
Since by definition, a geodesic has finite Euclidean length, this procedure terminates.
\end{defn}
\begin{figure}
\centering
\includegraphics[height=14cm]{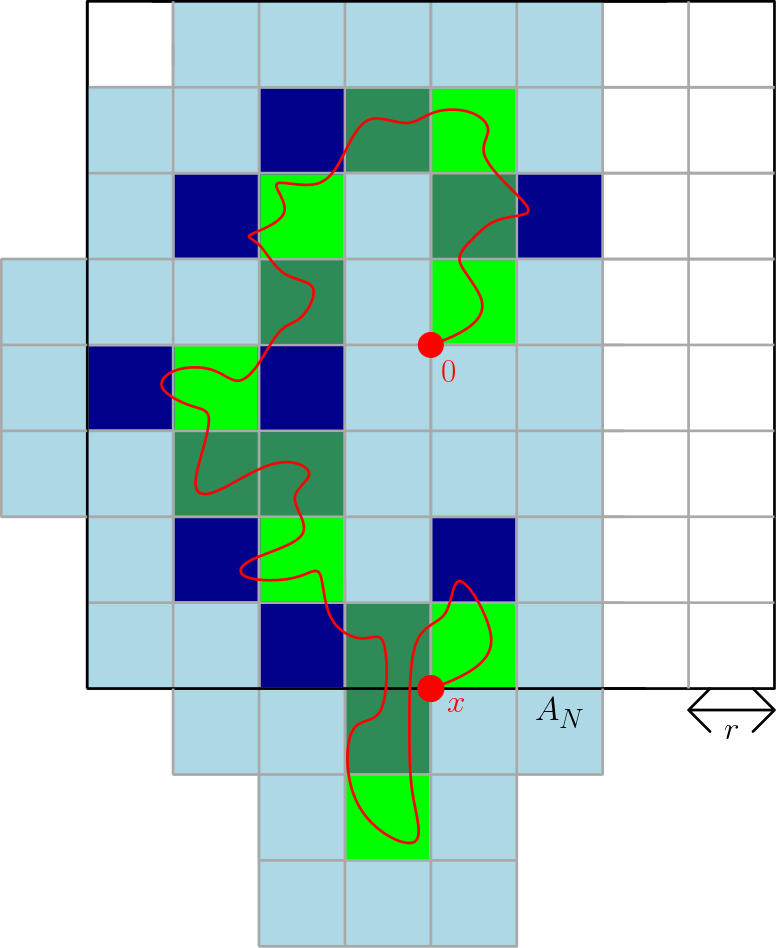}
\caption{In red, a geodesic $\gamma$ between $0$ and $x$. In light green, all squares in $\tilde{\mathcal{G}}_\gamma(0,x)$. In dark green, all other squares which are \emph{properly used}. In dark blue, all other squares the geodesic goes through. In light blue, all other squares in $\mathcal{G}(0,x)$.}
\label{f:squares}
\end{figure}
This construction is illustrated in Figure \ref{f:squares}. The fact that these squares are separated by a distance more than $r$ is obvious given the definition. The fact that they are at distance less than $2r$ can be seen thanks to Lemma \ref{proper}. 
\begin{rem}\label{ann}
For any square in $\tilde{\mathcal{G}}_\gamma(0,x)$, there exists an integer point at distance less than $1$ from its boundary and a copy of $A_{r/3^d}$ centered at that point which is crossed by the geodesic $\gamma$, as is clear given the definition of a \emph{properly used} square (Definition \ref{properlyused}). 
\end{rem}
For any geodesic $\gamma$, starting from $\tilde{\mathcal{G}}_\gamma(0,x)$, we define the set $\mathcal{G}_\gamma^\dagger(0,x)$ in the following way:
\begin{defn}\label{algo2}
Start with  $\mathcal{G}_\gamma^\dagger(0,x)=\tilde{\mathcal{G}}_\gamma(0,x)$.
\begin{itemize}
\item Add to the set $\mathcal{G}_\gamma^\dagger(0,x)$ all neighbors of its elements which are \emph{properly used}. We thus add no more than $3^d-1$ new squares per already present square.
\item Then add to the set $\mathcal{G}_\gamma^\dagger(0,x)$ all neighbors of its elements which the geodesic $\gamma$ goes through but are not \emph{properly used}. Again, no more than $3^d-1$ squares per previously present square are added.
\item Finally, add all neighbors of all previously present squares. Once more, no more than $3^d-1$ squares per previously present square are added.
\end{itemize}
\end{defn}
\begin{prop}\label{contained}
For any $r>0$, for any Gaussian field $f$, for any level $\ell$, for any $x\in\RR^d$, for any geodesic $\gamma$ between $0$ and $x$,
$$
\mathcal{G}_\gamma^\dagger(0,x)\supseteq \mathcal{G}(0,x).
$$
\end{prop}
\begin{proof}
The first step of the procedure in Definition \ref{algo2}  allows us to obtain a set of squares containing all \emph{properly used} squares, this is clear given Definition \ref{algo1}. The second step allows us to recover a set containing all the squares the geodesic goes through, as can be seen thanks to Lemma \ref{proper}.
Finally, the third step allows to recover all squares in the set $\mathcal{G}(0,x)$, which is clear given its definition.
Thus for any $\gamma$, $\mathcal{G}_\gamma^\dagger(0,x)\supseteq \mathcal{G}(0,x)$. 
\end{proof}
In Figure \ref{f:squares}, all intermediate values of the set $\mathcal{G}_\gamma^\dagger(0,x)$ are represented.
Combining all the estimates inside Definition \ref{algo2} and using Proposition \ref{contained}, we get the following.
\begin{lem}\label{combin}
For any $r>0$, for any $x$ of norm $N$, for any Gaussian field $f$ and associated pseudometric, for any $x\in\RR^d$, for any geodesic $\gamma$ between $0$ and $x$,
\begin{equation*}
\#\tilde{\mathcal{G}}_\gamma(0,x)\geq \frac{1}{(3^d)^3}\#\mathcal{G}_\gamma^\dagger(0,x)\geq \frac{1}{3^{3d}} \#\mathcal{G}(0,x).
\end{equation*}
\end{lem}
Finally, we will need the following control on the expectation of the pseudometric.
\begin{lem}\label{subaddlem}
Let $f$ be a Gaussian field satisfying Assumption \ref{a:basic}. Let $u=(u_n)_{n\in\NN}$ be a superadditive sequence such that for any $n$, $u_n\geq n$. Let $T$ be a pseudometric as in Definition \ref{pseudometric}, and for any $r>0$, $f_r$ be given by Definition \ref{d:cutoff}. Let $\ell\in\RR$. There exists a constant $C_{f,u}>0$ such that for any $n\in\NN$ and $x$ of norm $n$, for any $r\geq 1$,
$$
\frac{\EE_{r}T^{\mathcal{B}_{u_n}}(0,x)}{n}\leq C_{f,u},
$$
where $T^{\mathcal{B}_{u_n}}$ is as in (\ref{restrtime}).
\end{lem}
\begin{proof}
Let $f$ be a Gaussian field satisfying Assumption \ref{a:basic} and $\ell\in\RR$. Fix a vector $y$ of norm $1$. For any $r\geq 1$, the sequence
$$
(\EE_{r}T^{\mathcal{B}_{u_n}}(0,ny))_n
$$
is subadditive. 
Indeed, if $n$, $m$ are two integers,
\begin{align*}
& T^{\mathcal{B}_{u_{n+m}}}(0,(n+m)y)\\
&\quad \leq T^{\mathcal{B}_{u_{n+m}}}(0,ny)+T^{\mathcal{B}_{u_{n+m}}}(ny,(n+m)y)\\
&\quad \leq T^{\mathcal{B}_{u_{n}}}(0,ny)+T^{\mathcal{B}_{u_{m}}+ny}(ny,(n+m)y),
\end{align*}
where in the last step we have used the superadditivity of the sequence $(u_n)$ and the fact that for any $n$, $u_n\geq n$ to argue that
$$
\mathcal{B}_{u_{m}}+ny\subseteq \mathcal{B}_{u_{n}}+\mathcal{B}_{u_{m}}\subseteq \mathcal{B}_{u_{n+m}}.
$$
Thus, for any $r\geq 1$, for any $n\in\NN$,
\begin{equation}\label{usesubadd}
\frac{\EE_{r}T^{\mathcal{B}_{u_n}}(0,ny)}{n}\leq\EE_{r}T^{\mathcal{B}_{u_1}}(0,y),
\end{equation}
Further, by Definition \ref{pseudometric}, for any $r$, 
$$\EE_{r}T^{\mathcal{B}_{u_1}}(0,y)\leq u_1\sqrt{d}\text{ }\EE[\psi(\|f_r\|_{\infty,\mathcal{B}_{u_1}})]\leq C_\psi u_1\sqrt{d}\text{ } \EE\|f_r\|_{\infty,\mathcal{B}_{u_1}}.$$
 Thus by Corollary \ref{muirvansupnorm2} we deduce that there exists a constant $C_{f,u}$ such that for any $r\geq 1$,
$$
\EE_{r}T^{\mathcal{B}_{u_1}}(0,y)\leq C_{f,u}.
$$
Combining this with (\ref{usesubadd}) yields the conclusion.

\end{proof}

\subsubsection{Proof}

We are now ready to start the main proof. Once again, it is inspired by Kesten's original proof of the linear variance bound for classical first passage percolation (\cite{10.1214/aoap/1177005426}, equation 1.13 in Theorem 1), as well as its simplified version presented by Auffinger, Damron and Hanson in \cite{auffinger201750}.

\begin{proof}[Proof of Theorem \ref{main}] Let $f$ be a Gaussian field and $\ell$ be a level such that the assumptions are verified. Recall that for any $r>0$, $f_r$ is the field obtained from the same white-noise as $f$ but which has correlation range $r$ (Definition \ref{corrrange}), and for any level $\ell$, $\Var_{r}$, $\EE_{r}$, $\PP_{r}$ are defined for the corresponding measure. 
Notice that, by Proposition \ref{variancesfields}, it is enough to show that for any $\varepsilon>0$ there exists $C_\varepsilon>0$ such that for any $N\in \NN_{\geq 2}$ and $x$ of norm $N$
\begin{equation}\label{goal}
\Var_{(\log N)^{1/\alpha+\varepsilon}} T^{\mathcal{B}_{N^2}}(0,x)\leq C'_\varepsilon N(\log N)^{1/\alpha+\varepsilon}.
\end{equation}
Fix $\varepsilon>0$.
In the rest of this proof, all statements we make for some integer $N$ relate to the field $f_{(\log N)^{1/\alpha+\varepsilon}}$ and the pseudometric $T^{\mathcal{B}_{N^2}}$. For the sake of readability, we do not write the corresponding indices and exponents. For example, in this proof, when we write 
$$
\Var T (0,x),\quad\text{we mean}\quad\Var_{(\log N)^{1/\alpha+\varepsilon}} T^{\mathcal{B}_{N^2}} (0,x).
$$
Let $N\in\NN_{\geq 2}$ and $x$ be of norm $N$. 
Let $(S_i)_{i\in\NN}$ be the sequence of hyper-squares defined by the lattice $(\log N)^{1/\alpha+\varepsilon}\ZZ^d$ (see Definition \ref{squareslattice}), ordered in an arbitrary way.
We apply Efron-Stein's inequality (Proposition \ref{ES}) to $T(0,x)$:
$$
\Var T (0,x)\leq\sum\limits_{i}\EE\left[\left(T^{*i}(0,x)-T(0,x)\right)_+^2\right],
$$
where $T^{*i}(0,x)$ designates the random variable $T(0,x)$ where we have resampled  the white-noise in the square $S_i$.
Now, notice that since the field has correlation range $(\log N)^{1/\alpha+\varepsilon}$, $\Big(T^{*i}(0,x)-T(0,x)\Big)_+$ can be non-zero only if all of the geodesics defining $T(0,x)$ go at distance less than $(\log N)^{1/\alpha+\varepsilon}$ from the square $S_i$. Indeed, if some geodesic $\gamma$ goes at a larger distance from the square, then resampling the white-noise in the square does not change $T(\gamma)$, hence $T(0,x)$ does not increase. Otherwise stated, 
$$\Big(T^{*i}(0,x)-T(0,x)\Big)_+\neq 0 \implies S_i \in \mathcal{G}(0,x),$$
recalling Definition \ref{Gox}. Further, by Definition \ref{pseudometric}, we always have for any $N$, for any $x$ of norm $N$,
$$
\Big(T^{*i}(0,x)-T(0,x)\Big)_+\leq C_\psi (\log N)^{1/\alpha+\varepsilon} \|f^{*i}+\ell\|_{\infty, S_i^+},
$$
where $f^{*i}$ designates the non-stationary Gaussian field $q\star W|_{S_i}$, $W|_{S_i}$ being the resampled white-noise on $S_i$ and $0$ elsewhere, and $S_i^+$ designates the union of $S_i$ and all of its neighbors.
In particular, $\|f^{*i}+\ell\|_{\infty, S_i^+}$ depends only on the resampled white-noise in $S_i$. Hence,
\begin{align*}
\begin{split}
& \Var T (0,x)
\\
& \quad \leq \sum\limits_{i}  \EE\Bigg[ \Big(T^{*i}(0,x)-T(0,x)\Big)_+^2\un_{S_i\in \mathcal{G}(0,x)}\Bigg]
\\
&  \quad \leq C_\psi ^2(\log N)^{2/\alpha+2\varepsilon}\sum\limits_{i}  \EE\Bigg[ \|f^{*i}+\ell\|_{\infty, S_i^+}^2\un_{S_i\in \mathcal{G}(0,x)}\Bigg]
\\
& \quad = C_\psi ^2(\log N)^{2/\alpha+2\varepsilon}\sum\limits_{i}\EE\left[ \|f^{*i}+\ell\|_{\infty, S_i^+}^2\right]\EE\left[\un_{S_i\in\mathcal{G}(0,x)}\right].
\end{split}
\end{align*}
Now, using Corollary \ref{muirvansupnorm2}, there exists a constant $C_0$ such that for any $N\in\NN_{\geq 3}$ and for any $i$,
$$
\EE\left[ \|f^{*i}+\ell\|_{\infty, S_i^+}^2\right]\leq C_0(\log\log N)^2. 
$$
So that, for any $N$, for any $x$ of norm $N$,
$$
\Var T (0,x)\leq C_0C_\psi (\log N)^{2(1/\alpha+\varepsilon)}(\log\log N)^2 \EE \#\mathcal{G}(0,x).
$$
Fix some geodesic $\gamma$ from $0$ to $x$. Using Lemma \ref{combin} and recalling the set $\tilde{\mathcal{G}}_\gamma(0,x)$ from Definition \ref{algo1}, we deduce that for any $N\in\NN_{\geq 3}$, $x\in\RR^d$ such that $|x|=N$, 
\begin{equation}
\label{varmain}
\Var T (0,x)\leq  C_0C_\psi 3^{3d} (\log N)^{2(1/\alpha+\varepsilon)}(\log \log N)^2\EE \#\tilde{\mathcal{G}}_\gamma(0,x).
\end{equation}
Now, for any $N$, for any $x$ such that $|x|=N$, for any $a>0$, define the random variable $Y_N^a$ to be:
$$
Y_{N}^a:=\#\tilde{\mathcal{G}}_\gamma(0,x)\un\Big\{\sum\limits_{S_i\in\tilde{\mathcal{G}}_\gamma(0,x)}T(\mathcal{A}_{S_i})<a(\log N)^{1/\alpha+\varepsilon}\#\tilde{\mathcal{G}}_\gamma(0,x)\Big\},
$$
where for each $i$, $\mathcal{A}_{S_i}$ is the first annulus of side $(\log N)^{1/\alpha+\varepsilon} /3^d$ with center in $\ZZ^d$ and at distance less than one from the boundary of $S_i$ crossed by the geodesic. Such an annulus exists, as is justified by Remark \ref{ann}.
We thus write for any $a$, $N$, for any $x$ such that $|x|=N$,
\begin{equation}\label{twoparts2}
\EE \#\tilde{\mathcal{G}}_\gamma(0,x)\leq a^{-1}(\log N)^{-(1/\alpha+\varepsilon)}\EE T(0,x)+\EE Y_{N}^a.
\end{equation}
Definition \ref{algo1} allows for $\tilde{\mathcal{G}}(0,x)$ to fill the conditions of for separation of annuli in the family of events $\mathcal{E}_{a,B,r,N,M}$ (Definition \ref{E2}). We apply Lemma \ref{length}, remembering that we have assumed that Condition \ref{macrotime} is satisfied. We get $a>0$, $N_0\in\NN$ such that for any $N\geq N_0$, for any $M\geq 2\frac{N}{(\log N)^{1/\alpha+\varepsilon}}$,
$$
\PP[Y_{N}^a\geq M]\leq \left(\frac{1}{2}\right)^{M},
$$
so that, up to increasing $N_0$, for any $N\geq N_0$,
$$
\EE Y_N^a \leq 3\frac{N}{(\log N)^{1/\alpha+\varepsilon}}.
$$
Further, by Lemma \ref{subaddlem}, there exists a constant $C_1$ such that for any $N$, for any $x$ such that $|x|=N$,
$$
\EE T(0,x)\leq C_1 N.
$$
We deduce by (\ref{twoparts2}) that for any $N\geq N_0$, for any $x$ such that $|x|=N$,
\begin{equation*}
\EE \#\tilde{\mathcal{G}}_\gamma(0,x)\leq (a^{-1}C_1+3) \frac{N}{(\log N)^{1/\alpha+\varepsilon}}.
\end{equation*}
So that, reprising (\ref{varmain}),  for any $N\geq N_0$, for any $x$ such that $|x|=N$,
\begin{equation*}
\Var T (0,x)\leq C_0C_\psi 3^{3d}(a^{-1}C_1+3) N(\log N)^{1/\alpha+\varepsilon}(\log \log N)^2,
\end{equation*}
which yields, up to changing $\varepsilon$ to $\varepsilon/2$, a constant $C_\varepsilon>0$ depending on $\varepsilon$ such that for any $N\in\NN_{\geq 2}$ and $x$ of norm $N$
\begin{equation}\label{maincutoff}
\Var T (0,x)\leq C_\varepsilon N(\log N)^{1/\alpha+\varepsilon},
\end{equation}
which is (\ref{goal}).
\end{proof}

Finally, let us state the proposition used to deduce Corollary \ref{2dcoro} from Theorem \ref{main}.
\begin{prop}\label{numberballs}
Let $f$ be a Gaussian field on $\R^d$ satisfying  Assumption \ref{a:basic} with decay function $F$ $\alpha$-sub-exponential (see Definition \ref{subpoly}) for some $\alpha>1$ and Assumption \ref{a:pos}. Let $\ell>-\ell_c(f)$. There exist constants $a$, $C_1$ such that for any $N\geq 1$,
$$
\PP (T(A_N)\leq aN)\leq C_1e^{-N^{\alpha/5}}.
$$
In particular, Condition \ref{macrotime} is then verified.
\end{prop}

\section{Proof of auxiliary results}\label{proofaux}
In this section, we prove the toolbox results and all other auxiliary results from the previous section. 
\subsection{Variance comparison}
Let us start with everything that pertains to establishing the main variance comparison result, Proposition \ref{variancesfields}.
The following is a useful lemma for comparing a Gaussian field with its finite-correlation-range counterparts.
\begin{prop}[\cite{muirhead2018sharp}, Proposition 3.11]\label{muirvanapprox}
For any Gaussian field $f$ verifying Assumption \ref{a:basic} for some decay function $F$, there exist constants $C_1$, $C_2$, $r_0>0$ such that for all $N\in \NN$, $r\geq r_0$, for all $t\geq \log N$,
$$
\PP \left[\|f-f_r\|_{\infty, \mathcal{B}_N}\geq C_1 t F(r)\right]\leq e^{-C_2t^2}.
$$
\end{prop}
Muirhead and Vanneuville \cite{muirhead2018sharp} have proved the following using the previous Proposition.
\begin{prop}[\cite{muirhead2018sharp}, Proposition 4.1]\label{CMMV}
Consider a Gaussian field $f$ on $\RR^2$ satisfying Assumption \ref{a:basic} with moving average kernel $q$ and decay function $F$. Then there exists $c_1>0$ such that, for every $N\in \NN$ and $r\geq1$, every monotonic event $A$ measurable with respect to the field inside a ball of radius $N$ and every level $\ell$
$$
|\PP[A]-\PP_{r}[A]|\leq c_1\Big(N(\log N)F(r)+N^{-\log N}\Big).
$$
\end{prop}
We will now start on the variance control results, with a method inspired by Muirhead and Vanneuville's. We use the last two propositions to establish a similar estimate, but this time for variances:
\begin{lem}\label{CM}
For any Gaussian field $f$ satisfying Assumption \ref{a:basic} with moving average kernel $q$ and decay function $F$, there exist constants $C_0$, $C_1>0$, $r_0$ such that for any $N\in\NN$, $r\geq r_0$, for any pair of sets $A,B$ in a square of side $N$, for any $\ell$,
$$
\Big|\Var  (T^{\mathcal{B}_N}(A,B))-   \Var_{r} (T^{\mathcal{B}_N} (A,B))\Big| \leq C_0 \max\left(N^4(\log N)^5 F(r),N^{-C_1\log N}\right),
$$
$T^{\mathcal{B}_N}$ being the restricted pseudometric (see (\ref{restrtime})).
\end{lem}

To prove this lemma, we will need an intermediate result, Lemma \ref{Cs}. Let us make the following preliminary statement, whose elementary proof we omit.
\begin{lem}\label{varianceprop}
Let $X$ be a $L^2$ real-valued random variable. For any constant $c$, we have
$$
\EE[(X-c)^2] \geq \Var X.
$$
\end{lem}
%\begin{proof}
%Define the map $g$ on $\RR$ such that for any $x$, $g(x):=\EE[(X-x)^2].$
%We have for any $x$, $g(x)=\EE[X^2]-2x\EE[X]+x^2=\EE[X^2]+x(x-2\EE[X]).$
%Hence its minimum is reached at $\EE[X]$.
%\end{proof}

The auxiliary Lemma is the following.
\begin{lem}\label{Cs}
For any pair of Gaussian fields $f,g$ satisfying  Assumption \ref{a:basic} there exist $c>0$, $N_0\in\NN$ such that for any $N\geq N_0$, $t>0$ such that $Nt\leq c$, for any random variable $X(f)$ that is $L^4$, measurable with respect to the field $f$ in $\mathcal{B}_N$ and monotonic (see Definition \ref{increasing}),
$$
\Var (X(f)\mathcal{E}_t)\leq  \Var [X(g)] + \PP[\mathcal{E}_t=0]\EE[X(g)]^2 +\frac{2Nt}{\int\limits_{\RR^d} q} \left(\EE\left[\Big(X(g)-\EE[X(g)]\Big)^4\right]\right)^{1/2},
$$
where $\mathcal{E}_t$ is the random variable $\un_{\|f-g\|_{\infty,\mathcal{B}_N}\leq t}$ and $q$ is the moving-average kernel of $g$.
Further, there exists $r_0>0$ such that for all $g=f_r$, $r\geq r_0$, the bound holds with the value of $N_0$ being identical and the kernel $q$ being that of $f$.
\end{lem}
\begin{proof}
Let $f$ and $g$ be as in the statement. Through a Cameron-Martin construction presented in the Appendix (see  section \ref{p:CM}) applied to $g$, we can find a function $h$ such that $|h|\geq 1$ on $\mathcal{B}_N$ and there exists a random variable $Q(h)$ called the Radon-Nikodym difference associated to $h$ such that for any event $A$, by Proposition \ref{CMcontrol}
\begin{equation}\label{controldifference}
|\PP[g\in A]-\PP[g+h\in A]|= \left|\EE_g\left[Q(h)\un_A\right]\right|.
\end{equation}
By Proposition \ref{CMcontrol2} there exist universal constants $c$, $C_0>0$ and constants $N_0$ depending on $g$ such that for any $N\geq N_0$, for any $t\leq c/N$,
\begin{equation}\label{q}
\EE [Q(th)^2]\leq \frac{C_0tN}{\int q},
\end{equation}
where
$$N_0=\inf \{N\in\NN, \quad \inf\limits_{\mathcal{B}_{\frac1N}} \rho\geq \frac{\rho(0)}{2}\}$$
and $q$ is the moving-average kernel of $g$.
Further, if $g=f_r$, if we call $\rho$ the spectral density of $f$ and for any $r$, $\rho_r$ that of $f_r$, we have $\rho_r\overset{a.e}{\underset{r\to\infty}{\longrightarrow}} \rho$. Thus there exist universal constants $c$, $C_0>0$ and constants $N_0$ depending on $f$ such that for any $N\geq N_0$, $r\geq r_0$, for any $t\leq c/N$,
\begin{equation}\label{q'}
\EE [Q_r(th)^2]\leq \frac{C_0tN}{\int q},
\end{equation}
$Q_r$ being the Radon-Nikodym difference for $f_r$ and $q$ being the moving-average kernel of $f$.

Now, consider an integer $N$ and a number $0<t\leq c/N$. Consider a monotonic $L^4$ random variable $X$. We suppose without loss of generality that it is increasing.
By Lemma \ref{varianceprop}, it is enough to bound $\EE[(X(f)\mathcal{E}_t-\EE [X(g)])^2]$.
We thus write
\begin{align}
\begin{split}
\label{beforeCM}
&
\EE[(X(f)\mathcal{E}_t-\EE[ X(g)])^2]
\\
&\quad = \int\limits_{\RR_+} \PP[(X(f)\mathcal{E}_t-\EE[ X(g)])^2\geq u]du
\\
&\quad = \int\limits_{\RR_+} \left[\PP[X(f)\mathcal{E}_t-\EE[ X(g)]\geq \sqrt{u}]+\PP[X(f)\mathcal{E}_t-\EE[ X(g)]\leq -\sqrt{u}]\right]du
\\
&\quad = \int\limits_{\RR_+} \left[\PP[X(g+th)\mathcal{E}_t-\EE[ X(g)]\geq \sqrt{u}]+\PP[X(g-th)\mathcal{E}_t-\EE[ X(g)]\leq -\sqrt{u}]\right]du,
\end{split}
\end{align}
where in the last step we have used that when $\mathcal{E}_t\neq 0$, on $\mathcal{B}_N$,
$$
g-th\leq f\leq g+th,
$$
and increasingness of $X$. Now, we always have $X\mathcal{E}_t\leq X$ and further, when $u> \EE [X(g)]^2$, we have
$$
\PP[X(g-th)\mathcal{E}_t-\EE[ X(g)]\leq -\sqrt{u}]\leq \PP[X(g-th)-\EE[ X(g)]\leq -\sqrt{u}].
$$
When $u\leq \EE [X(g)]^2$ , we bound this quantity by 
$$
 \PP[X(g-th)-\EE[ X(g)]\leq -\sqrt{u}]+\PP[\mathcal{E}_t=0].
$$
So that, retunring to (\ref{beforeCM}), we have
\begin{align*}
&
\EE[(X(f)\mathcal{E}_t-\EE[ X(g)])^2]
\\
&\quad \leq \int\limits_{\RR_+} \left[ \PP[X(g-th)-\EE[ X(g)]\geq \sqrt{u}]+ \PP[X(g-th)-\EE[ X(g)]\leq -\sqrt{u}]\right]du
\\
&\quad + \EE [X(g)]^2\PP[\mathcal{E}_t=0].
\end{align*}
We recall (\ref{controldifference}) and get
\begin{align*}
& 
\EE[(X(f)\mathcal{E}_t-\EE[ X(g)])^2]
\\
& \quad \leq \int\limits_{\RR_+} \left[\PP[X(g)-\EE[X(g)]\geq \sqrt{u}]+\PP[X(g)-\EE[X(g)]\leq -\sqrt{u}]\right]du + \EE [X(g)]^2\PP[\mathcal{E}_t=0].
\\
& \quad + \int\limits_{\RR_+}\left[\EE[Q(th)\un_{X(g)-\EE[X(g)]\geq \sqrt{u}}]+\EE[Q(-th)\un_{X(g)-\EE[X(g)]\leq -\sqrt{u}}]\right]
\\
& \quad \leq \Var[X(g)]+\int\limits_{\RR_+}\EE[(|Q(th)|+|Q(-th)|)\un_{(X(g)-\EE[X(g)])^2\geq u}]+ \EE [X(g)]^2\PP[\mathcal{E}_t=0],
\end{align*}
$Q$ being the Radon-Nikodym difference of $g$.
So that by the Cauchy-Schwarz inequality, this can be bounded for any $N\geq N_0$, for any $t$ small enough by
\begin{align*}
& \Var [X(g)]+\Big(\EE\left[(|Q(th)|+|Q(-th)|)^2\right]\EE\Big[\Big(\int\limits_{\RR_+}\un_{(X(g)-\EE[X(g)])^2\geq u}du\Big)^2\Big]\Big)^{1/2}+ \EE [X(g)]^2\PP[\mathcal{E}_t=0]
\\
& =\Var [X(g)]+\left(\EE\left[(|Q(th)|+|Q(-th)|)^2\right]\EE\left[(X(g)-\EE [X(g)])^4\right]\right)^{1/2}+ \EE[X(g)]^2\PP[\mathcal{E}_t=0]
\\
& \leq \Var [X(g)]+\frac{C_0N}{\int q}t\left(\EE\left[(X(g)-\EE[X(g)])^4\right]\right)^{1/2}+ \EE [X(g)]^2\PP[\mathcal{E}_t=0].
\end{align*}
where we have used relation (\ref{q}), (resp (\ref{q'}) for the statement with $g=f_r$).
\end{proof}

We can now prove Lemma \ref{CM}.

\begin{proof}[Proof of Lemma \ref{CM}]
Let $f$ be a Gaussian field satisfying the assumptions of Lemma \ref{CM}. Let $\ell\in\RR$. For clarity, in this proof, we use the index $r$ directly on the pseudometric $T$ to indicate which field we work with instead of on the variance operator, which may involve a random variable that depends on both $f$ and $f_r$. We prove that there exist constants $C_0$, $C_1$ such that for any $N\in\NN_{\geq 2}$
\begin{equation}\label{oneineq} 
\Var (T^{\mathcal{B}_N}(A,B))\leq  \Var (T_{r}^{\mathcal{B}_N} (A,B)) + C_0 \max\left(N^4(\log N)^5 F(r),N^{-C_1\log N}\right),
\end{equation}
the other inequality's proof is identical, reversing the roles of the two fields.
 Let $C_1$ be the constant from Proposition \ref{muirvanapprox}. For any $N$,$r$, $t$, let $\mathcal{E}_{N,r,t}$ be the event 
$$\mathcal{E}_{N,r,t}:=\{\|f-f_r\|_{\infty, \mathcal{B}_N}\geq C_1 t F(r)\}.$$
%Let $T_{\bigcdot}$ signify that the pseudometric is taken either for $f$ or for $f_r$, the computation being valid for both, and likewise let $f_{\bigcdot}$ denote either $f$ or $f_r$. 
We have, for all $N,r$, $A,B\subseteq\mathcal{B}_N$, for all $t$, for all $\ell$,
\begin{align}
\begin{split}
\label{twoparts}
&\Var T^{\mathcal{B}_N}(A,B)
\\
&\quad =\Var\Big[ T^{\mathcal{B}_N}(A,B)\un_{\mathcal{E}_{N,r,t}^\compl}
+T^{\mathcal{B}_N}(A,B)\un_{\mathcal{E}_{N,r,t}}\Big]
\\
&\quad =\Var\Big[ T^{\mathcal{B}_N}(A,B)\un_{\mathcal{E}_{N,r,t}^\compl}\Big]+\Var\Big[ T^{\mathcal{B}_N}(A,B)\un_{\mathcal{E}_{N,r,t}}\Big]
\\
&\quad  +2\cov\left(T^{\mathcal{B}_N}(A,B)\un_{\mathcal{E}_{N,r,t}^\compl}, T^{\mathcal{B}_N}(A,B)\un_{\mathcal{E}_{N,r,t}}\right).
\end{split}
\end{align}
Let us treat the second term. 
Recall Definition \ref{pseudometric}. It implies that any pseudodistance between two sets $A$, $B$ within a ball of size $N$ is upper bounded by $C_\psi  N\|f\|_{\infty,\mathcal{B}_N}$. We use Propositions \ref{muirvansupnorm} and \ref{muirvanapprox} to control the two probabilities. For any $N\in \NN$, $r>0$, $t\geq\log N$, for any $A,B$:
\begin{align*}
\begin{split}
& \Var\Big[T^{\mathcal{B}_N}(A,B)\un_{\mathcal{E}_{N,r,t}}\Big]\leq C_\psi ^2N^2\EE[\|f+\ell\|_{\infty,\mathcal{B}_N}^2\un_{\mathcal{E}_{N,r,t}}]
\\
& \quad \leq C_\psi ^2N^2 \int\limits_{\RR_+} \min\Big(\PP\Big[\|f+\ell\|_{\infty,\mathcal{B}_N}^2\geq u \Big],\PP[\mathcal{E}_{N,r,t}]\Big)du
\\
& \quad \leq  C_\psi ^2N^2 \int\limits_{\RR_+} \min(\un_{u\leq \log N}+e^{-u^2/c_0}\un_{u\geq\log N},e^{-C_2t^2})du
\\
& \quad \leq  C_3 N^2\log Ne^{-C_4t^2},
\end{split}
\end{align*}
for some constants $C_3, C_4$  independent of $r$.
Returning to (\ref{twoparts}), for all $N$, $r$, $t$, for any $A,B$
\begin{align}\label{whenwestay}
\begin{split}
&\Var T^{\mathcal{B}_N}(A,B)\leq \Var\Big[ T^{\mathcal{B}_N}(A,B)\un_{\mathcal{E}_{N,r,t}^\compl}\Big]+\Var\Big[ T^{\mathcal{B}_N}(A,B)\un_{\mathcal{E}_{N,r,t}}\Big]
\\
&\quad + 2 \left(\Var\Big[ T^{\mathcal{B}_N}(A,B)\un_{\mathcal{E}_{N,r,t}^\compl}\Big]\Var\Big[ T^{\mathcal{B}_N}(A,B)\un_{\mathcal{E}_{N,r,t}}\Big]\right)^{1/2}
\\
& \quad \leq \Var\Big[ T^{\mathcal{B}_N}(A,B)\un_{\mathcal{E}_{N,r,t}^\compl}\Big]\Big(1+3(C_3\log N)^{1/2} Ne^{-\frac{C_4t^2}{2}} \Big),
\end{split}
\end{align}
assuming $\Var\Big[ T^{\mathcal{B}_N}(A,B)\un_{\mathcal{E}_{N,r,t}^\compl}\Big]$ is larger than $1$ (in the opposite case one can write the term $(C_3\log N)^{1/2} Ne^{-\frac{C_4t^2}{2}}$ without the factor).
%Likewise,
%\begin{align}\label{whenwestay2}
%\Var\Big[ T_{\bigcdot}^{\mathcal{B}_N}(A,B)\un_{\mathcal{E}_{N,r,t}^\compl}\Big]\leq \Var T_{\bigcdot}^{\mathcal{B}_N}(A,B)\Big(1+3(C_3\log N)^{1/2} Ne^{-\frac{C_4t^2}{2}} \Big).
%\end{align}

Now, apply Lemma \ref{Cs}. We bound the expectation of $T^{\mathcal{B}_N}(A,B)$ using Lemma \ref{subaddlem}. We bound the fourth moment by $\EE_r[(NC_\psi \|f_r\|_{\infty,\mathcal{B}_N})^4]$, itself bounded using Corollary \ref{muirvansupnorm2}. We get $N_0\in\NN$, $C_5, C_6$,  $r_0>0$ such that for any $N\geq N_0$, $r\geq r_0$, for any $t\geq \log N$, for any $A,B$
\begin{equation}\label{whenwestay3}
 \Var \Big[ T^{\mathcal{B}_N}(A,B)\Big]  \leq  \Var\Big[ T_r^{\mathcal{B}_N}(A,B)\un_{\mathcal{E}_{N,r,t}^\compl}\Big]+ C_5 N^2e^{-C_2t^2}+ C_6C_\psi^4 N^4(\log N)^4 tF(r).
\end{equation}
Combining (\ref{whenwestay}), then (\ref{whenwestay3}), we get for any $N\geq N_0$, $r\geq r_0$, for any $t\geq \log N$, for any $A,B$
\begin{align*}
\begin{split}
&\Var T^{\mathcal{B}_N}(A,B)\\
&\quad \leq \Var \Big[ T^{\mathcal{B}_N}(A,B)\un_{\mathcal{E}_{N,r,t}^\compl}\Big]\Big(1+3(C_3\log N)^{1/2} Ne^{-\frac{C_4t^2}{2}} \Big)\\
&\quad  \leq \left(  \Var\Big[ T_r^{\mathcal{B}_N}(A,B)\Big]+C_5 N^2e^{-C_2t^2}+ C_6C_\psi^4 N^4(\log N)^4 tF(r)\right)\left(1+3(C_3\log N)^{1/2} Ne^{-\frac{C_4t^2}{2}}\right).
\end{split}
\end{align*}

Take $t=\log N$ and get  constants $C_7$, $C_8$ such that for any $N\in \NN_{\geq 2}$, $r\geq 1$,
$$
\Var T^{\mathcal{B}_N}(A,B)\leq \left(  \Var T_r^{\mathcal{B}_N}(A,B) + C_6C_\psi^4 N^4(\log N)^5F(r)\right)\left(1+C_7N^{-C_8\log N}\right),
$$
which proves (\ref{oneineq}).
\end{proof}
The following technical lemma is combined with Lemma \ref{CM} in our computations to get our final variance comparison result.
\begin{defn}
For any $\varepsilon>0$, for any point $x$ such that $|x|=N$, define $\Gamma(0,x)$ to be a geodesic between $0$ and $x$ with minimal Euclidean diameter (chosen with some arbitrary rule).
\end{defn}
\begin{lem}\label{length2}
Let $f$ be a Gaussian field on $\RR^d$ verifying Assumption \ref{a:basic} with decay $F$, as well as $\ell$ such that conditon \ref{macrotime} is verified for some function $G$. There exists a constant $c_0>0$ such that for any $N\geq 1$, for any $x$ such that $|x|=N$ and $M\geq N^2$,
$$
\PP\Big[\Diam (\Gamma(0,x))\geq M\Big]\leq G(M) + e^{-M/c_0}.
$$
\end{lem}
\begin{proof}
We have, for any $N$, for any $x$ such that $|x|=N$, $M\geq N^2$, for any $r$
\begin{equation}\label{macrotimes1}
\PP\Big[\Diam (\Gamma(0,x))\geq M\Big]\leq \PP[ T(A_{M})\leq \|f+\ell\|_{\infty, \mathcal{B}_N}N].
\end{equation}
Indeed, if $\Diam (\Gamma(0,x))\geq M$ then all geodesics exit the annulus $A_{M}$, and do so with time smaller than $ \|f\|_{\infty, \mathcal{B}_N}N$, otherwise the euclidean geodesic between $0$ and $x$ would have smaller time.
Now, by Proposition \ref{muirvansupnorm} for any $a>0$, there exists a constant $c_0>0$ such that for any $N\in\NN$, for any $M\geq N^2$,
\begin{equation}\label{macrotimes2}
\PP[\|f+\ell\|_{\infty, \mathcal{B}_N}N\geq aM]\leq e^{-\frac{M^2}{c_0N^2}}.
\end{equation}
Now, since $M\geq N^2$, $\frac{M^2}{N^2}\geq M$.
And 
$$
\PP[ T(A_{M})\leq \|f+\ell\|_{\infty, \mathcal{B}_N}N]\leq \PP[\|f+\ell \|_{\infty, \mathcal{B}_N}N\geq aM]+\PP[ T(A_{M})\leq aM].
$$
So that, by combining (\ref{macrotimes1}), (\ref{macrotimes2}) and Condition \ref{macrotime}, we get the desired result.
\end{proof}
Now, on to the main proof of this subsection.
\begin{proof}[Proof of Proposition \ref{variancesfields}]
Once again, we only prove one inequality, the other one's proof being identical. Fix a field $f$ and a level $\ell$.
For any $N$, for any $x$ of norm $N$ and $M>N$, recall that $\Gamma(0,x)$ is the intersection of all geodesics between $0$ and $x$ and let
\begin{equation}\label{smallevent}
\mathcal{E}_{N,x,M} :=\{\Diam\Gamma(0,x) \leq M\}.
\end{equation} 
So that on that event, $T^{\mathcal{B}_M}(0,x)$ and $T(0,x)$ coincide and we have
\begin{align}\label{cov}
\begin{split}
& \Var T(0,x)
\\
& \quad \leq \Var [T^{\mathcal{B}_M}(0,x)\un_{\mathcal{E}_{N,x,M}}]+\Var [T(0,x)\un_{\mathcal{E}_{N,x,M}^\compl}]
\\
& \quad +2\left(\Var [T^{\mathcal{B}_M}(0,x)\un_{\mathcal{E}_{N,x,M}}]\Var[T(0,x)\un_{\mathcal{E}_{N,x,M}^\compl}]\right)^{1/2}.
\end{split}
\end{align}
And likewise,
\begin{align}\label{cov2}
\begin{split}
& \Var [T^{\mathcal{B}_M}(0,x)\un_{\mathcal{E}_{N,x,M}}]
\\
& \quad \leq \Var [T^{\mathcal{B}_M}(0,x)]+\Var [T^{\mathcal{B}_M}(0,x)\un_{\mathcal{E}_{N,x,M}^\compl}]
\\
& \quad +2\left(\Var [T^{\mathcal{B}_M}(0,x)]\Var[T^{\mathcal{B}_M}(0,x)\un_{\mathcal{E}_{N,x,M}^\compl}]\right)^{1/2}.
\end{split}
\end{align}
%Now, for any $\ell, N,M$ and $x$ of norm $N$, 
%\begin{align}
%\begin{split}
%\label{interior}
%& \Var  [T(0,x)\un_{\mathcal{E}_{N,x,M}}]
%\\
%& \quad =\Var  [T^{\mathcal{B}_M}(0,x)\un_{\mathcal{E}_{N,x,M}}]
%\\
%& \quad \leq \EE [(T^{\mathcal{B}_M}(0,x)\un_{\mathcal{E}_{N,x,M}}-\EE [T^{\mathcal{B}_M}(0,x)])^2]
%\\
%& \quad \leq \Var  [T^{\mathcal{B}_M}(0,x)]+\PP [\mathcal{E}_{N,x,M}^\compl]\EE [T^{\mathcal{B}_M}(0,x)]^2
%\end{split}
%\end{align}
%where on the third line we have used Propositon \ref{varianceprop}. 
%We also have for any $\ell, N,M,x$
%\begin{align}
%\begin{split}
%\label{interior2}
%& \Var_{(\log N)^{1/\alpha+\varepsilon}}[T^{\mathcal{B}_M}(0,x)]
%\\
%& \quad =\Var_{(\log N)^{1/\alpha+\varepsilon}}[T(0,x)\un_{\mathcal{E}_{N,x,M}}+T^{\mathcal{B}_M}(0,x)\un_{\mathcal{E}_{N,x,M}^\compl}]
%\\
%& \quad =\Var_{(\log N)^{1/\alpha+\varepsilon}}[T(0,x)\un_{\mathcal{E}_{N,x,M}}+\Var[T^{\mathcal{B}_M}(0,x)\un_{\mathcal{E}_{N,x,M}^\compl}]
%\\
%& \quad +2\cov_{(\log N)^{1/\alpha+\varepsilon}}(T(0,x)\un_{\mathcal{E}_{N,x,M}},T^{\mathcal{B}_M}(0,x)\un_{\mathcal{E}_{N,x,M}^\compl})
%\end{split}
%\end{align}
Now, using Lemma \ref{CM}, we get  constants $C_0$, $C_1$ such that for any $N$, for any $M\geq N^2$ and $x$ such that $|x|=N$,
\begin{align}
\begin{split}
\label{restriction}
&\Var [T^{\mathcal{B}_M}(0,x)]
\\
& \quad \leq \Var_{(\log N)^{1/\alpha+\varepsilon}}[T^{\mathcal{B}_M}(0,x)]
\\
& \quad +C_0 \max\left(M^4(\log M)^5 F((\log N)^{1/\alpha+\varepsilon})N^{-C_1\log N}\right).
\\
\end{split}
\end{align}
Further, recalling Definition \ref{pseudometric}, there exists a constant $C_\psi $ such that for any $N,M$, for any $x$ of norm $N$,
\begin{align*}
\begin{split}
& \Var [T(0,x)\un_{\mathcal{E}_{N,x,M}^\compl}]
\\
& \quad \leq C_\psi  \EE[\|f+\ell\|^2_{\infty,\mathcal{B}_{\Diam(\Gamma(0,x))}}\Diam^2(\Gamma(0,x))\un_{\mathcal{E}_{N,x,M}^\compl}]\\
& \quad\leq  C_\psi  \Big( \EE[\Diam^4(\Gamma(0,x))\un_{\|f+\ell\|_{\infty,\mathcal{B}_{\Diam(\Gamma(0,x))}}\leq \Diam(\Gamma(0,x))}\un_{\mathcal{E}_{N,x,M}^\compl}] 
\\
& \quad +  \EE\big[\|f+\ell\|^4\un_{\mathcal{E}_{N,x,M}^\compl}\un_{\|f+\ell\|_{\infty,\mathcal{B}_{\Diam(\Gamma(0,x))}}\geq \Diam(\Gamma(0,x))}\big] \Big)
\\
&\quad \leq C_\psi\Big(\int\limits_{s\in\RR_+}\PP\big[\Diam^4(\Gamma(0,x))\un_{\mathcal{E}_{N,x,M}^\compl}\geq s\big]ds
\\
&\quad  + \int\limits_{s\in\RR_+}\PP\big[\|f+\ell\|^4_{\infty,\mathcal{B}_{\Diam(\Gamma(0,x))}}\un_{\mathcal{E}_{N,x,M}^\compl}\un_{\|f+\ell\|_{\infty,\mathcal{B}_{\Diam(\Gamma(0,x))}}\geq \Diam(\Gamma(0,x))}\geq s\big]ds \Big).
\end{split}
\end{align*}
We notice that in both previous integrals, for $s>0$, we can bound the integrand by $\PP[\mathcal{E}_{N,x,M}^\compl]$. We do so for $s\in(0,M^4]$.
So that
\begin{align}
\begin{split}
\label{smallterm}
& \Var [T(0,x)\un_{\mathcal{E}_{N,x,M}^\compl}]
\\
&\quad \leq C_\psi\Big( 2M^4\PP[\mathcal{E}_{N,x,M}^\compl]+\int\limits_{s\geq M^4}\PP\big[\Diam^4(\Gamma(0,x))\un_{\mathcal{E}_{N,x,M}^\compl}\geq s\big]ds
\\
&\quad  + \int\limits_{s\geq M^4}\PP\big[\|f+\ell\|^4_{\infty,\mathcal{B}_{\Diam(\Gamma(0,x))}}\un_{\mathcal{E}_{N,x,M}^\compl}\un_{\|f+\ell\|_{\infty,\mathcal{B}_{\Diam(\Gamma(0,x))}}\geq \Diam(\Gamma(0,x))}\geq s\big]ds \Big).
\end{split}
\end{align}
Notice that 
$$\PP\Big[\|f+\ell\|^4_{\infty,\mathcal{B}_{\Diam(\Gamma(0,x))}}\un_{\mathcal{E}_{N,x,M}^\compl}\un_{\|f+\ell\|_{\infty,\mathcal{B}_{\Diam(\Gamma(0,x))}}\geq \Diam(\Gamma(0,x))}\geq s\Big]\leq \PP\big[\|f+\ell\|^4_{\infty,\mathcal{B}_{s}}\geq s\big].$$
So that, applying Proposition \ref{muirvansupnorm} in (\ref{smallterm}),  for any $N$, for any $M> N$, for any $x$ of norm $N$,
\begin{align}\label{smallterm2}
\begin{split}
& \Var [T(0,x)\un_{\mathcal{E}_{N,x,M}^\compl}]
\\
& \quad \leq C_\psi  \Big( 2M^4\PP[\mathcal{E}_{N,x,M}^\compl]+\int\limits_{s\geq M^4} \PP[\mathcal{E}_{N,x,\varepsilon,s^{1/4}}^\compl]ds+\int\limits_{s\geq M^4}e^{-s^{1/2}/c_0}ds\Big)
\\
& \quad
\leq C_\psi \Big( 2M^4\PP[\mathcal{E}_{N,x,M}^\compl]+\int\limits_{s\geq M} 4s^3\PP[\mathcal{E}_{N,x,\varepsilon,s}^\compl]ds+C_1 e^{-M}\Big),
\end{split}
\end{align}
for some constant $C_1$.
And similarly,
\begin{align}\label{smallterm3}
\begin{split}
& \Var [T^{\mathcal{B}_M}(0,x)\un_{\mathcal{E}_{N,x,M}^\compl}]
\\
& \quad \leq C_\psi M^2\int\limits_{s\in\RR_+}\PP\big[\|f+\ell\|^2_{\infty,\mathcal{B}_{\Diam(\Gamma(0,x))}}\un_{\mathcal{E}_{N,x,M}^\compl}\geq s\big]ds
\\
& \quad \leq C_\psi \Big( M^4\PP[\mathcal{E}_{N,x,M}^\compl]+C_1 e^{-M}\Big).
\end{split}
\end{align}
Now, by Lemma \ref{length2}, recalling (\ref{smallevent}), there exists a $\max(2d,4)$-sub-polynomial function $G$ such that, for any $N$ and $M\geq N^2$,
\begin{equation}\label{largediam}
\PP[\mathcal{E}_{N,x,M}^\compl]\leq G(M),
\end{equation}
where the function $G$ includes both the function $G$ of the lemma and the exponential term.
So that we return to (\ref{cov}), use (\ref{cov2}) to remove the indicator of $\mathcal{E}_{N,x,M}$ and then use (\ref{restriction}) to switch from the field $f$ to $f_{(\log N)^{1/\alpha+\varepsilon}}$, and finally use (\ref{smallterm2}), (\ref{smallterm3}) and (\ref{largediam}) to control the error terms. We get a constant $C_2$ such that for all $N\in\NN$, for all $x$ such that $|x|=N$ and $M\geq N^2$,
\begin{align*}
&\Var T(0,x) 
\\
& \quad \leq\Big[\Var_{(\log N)^{1/\alpha+\varepsilon}}[T^{\mathcal{B}_M}(0,x)]
\\
& \quad +C_2 \max\left(M^4(\log M)^5 F((\log N)^{1/\alpha+\varepsilon}),N^{-C_1\log N}\right)\Big]
\\
& \quad \times \Big[1+6C_\psi \Big(3M^4G(M)+\int\limits_{s\geq M}4s^3G(s)ds +C_1e^{-M}\Big)^{1/2}\Big].
\end{align*}
Take $M=N^2$ and, recall that $F$ is $\alpha$\emph{-sub-exponential} for some $\alpha>1$ and $G$ is $4$\emph{-sub-polynomial} (see Definition \ref{subpoly}). We then have, for any $N$ and $x$ such that $|x|=N$,
$$
\Var T(0,x) \leq \Var_{(\log N)^{1/\alpha+\varepsilon}}\Big[T^{\mathcal{B}_{N^2}}(0,x)\Big]\big(1+o(1)\big)+o(1),
$$
where $o(1)$ are sequences depending only on $N$ and going to $0$ as $N$ goes to infinity.
\end{proof}
\subsection{On Condition \ref{macrotime}}
To prove Proposition \ref{numberballs}, the key proposition in establishing that Condition \ref{macrotime} is verified, we need the following auxiliary results.
\begin{lem}
\label{expdecann}
Let $f$ be a Gaussian field on $\RR^d$ satisfying Assumption \ref{a:basic} for some $\alpha$-sub-exponential function $F$ with $\alpha>0$, as well as the positivity assumption, Assumption \ref{a:pos}. Let $\ell>-\ell_c$.
	Then, there exist positive constants $c,M_0$ such that 
	$$
	\forall M\geq M_0, \exists a>0  \text{ such that}\ 	\PP\left[\frac{T(A_M)}{M}\leq a\right]\leq e^{-cM}.
	$$
\end{lem}
\begin{proof}
This follows readily from \cite[Theorem 1.2]{severo2021sharp} which gives exponential decay for annulus crossing probabilities, combined with right-continuity of the map $x\mapsto\PP[X\leq x]$ for any real-valued random variable $X$.
\end{proof}
%\begin{lem}
%There exists a constant $C_1$ such that for any Gaussian field $f$ verifying ... for some decay function $F$, there exist $C_2$, $N_0$, $r_0>0$ such that for all $N\geq N_0$, $r\geq r_0$, for all $t\geq \log N$, for any monotonic event $A$ depending on the field $f$ inside $\mathcal{B}_N$,
%$$
%|\PP [f\in A]-\PP [f_r\in A]|\leq \frac{C_1}{\int\limits_{\RR^d} q} t F(r)+e^{-C_2t^2}.
%$$
%\end{lem}
%\begin{proof}
%It is a direct consequence of Proposition \ref{muirvanapprox}.
%\end{proof}
\begin{lem}[\cite{dewangayet}, Proposition 4.4 and Corollary 5.9]
\label{bootstrap}
Let $f$ be a Gaussian field on $\RR^d$ satisfying  Assumption \ref{a:basic} for some $\alpha$-sub-exponential decay function $F$ with $1<\alpha<2$. Let $\ell\in\RR$. Then for any $1\leq Q<R<S$ and any positive constant $\delta$,
\begin{equation}\label{strass} 
\PP \left[\frac{T(A_S)}S< \frac{\delta}{1+\frac{Q}{R}}\right]\leq  
\left(c_d S^{d-1}\frac{R}{Q}\right)^n
\left(  \PP \left[\frac{T(A_R)}R< \delta \right]^{n}
+n Se^{-\frac12Q^{\alpha}}
%+ nS^2\exp(-\frac{1}2 Q^{\beta})
\right),
\end{equation}
where $c_d>0$ is a constant depending only on the dimension $d$,  where 
 $ n=\lfloor N\frac{Q}{2R+2Q}\rfloor $  with $N= \lfloor \frac{S-1}{2R+Q}\rfloor $.
\end{lem}
\bigskip 
Now on to the proof of Proposition \ref{numberballs}.
\begin{proof}Fix a Gaussian field $f$ verifying the assumptions.
Let $M_0$ and $c$ be the constants from Lemma \ref{expdecann}.
For all $k\in \NN$, and $M\in[M_0,M_0^2]$, let 
\begin{equation}\label{scales}
N_{M,k}:=M^{2^k}.
\end{equation}
Thus, note that for all $x\in [M_0,+\infty)$, there exist $M,k$ such that $x=N_{M,k}$.
Now, fix some $M\in[M_0,M_0^2].$ For any $k\geq 1$ and $\delta>0$, we have, by Lemma \ref{bootstrap} applied with $S=N_{M,k}$, $R=N_{M,k-1}$ and $Q=\sqrt{N_{M,k-1}}$ ($=N_{M,k-2}$ if $k\geq 2$),
\begin{align}\label{usebootstrap}
\begin{split}
&\PP\left[\frac{T(A_{N_{M,k}})}{N_{M,k}}< \frac{\delta}{1+\frac{1}{\sqrt{N_{M,k-1}}}}\right]
\\
& \leq \left(c_d N_{M,k}^{d-1}\sqrt{N_{M,k-1}}\right)^{\sqrt{N_{M,k-1}}}
\left(  \PP\left[\frac{T(A_{N_{M,k-1}})}{N_{M,k-1}}< \delta \right]^{\sqrt{N_{M,k-1}}}
+\sqrt{N_{M,k-1}} N_{M,k}e^{-\frac12N_{M,k-1}^{\alpha/2}}
\right).
\end{split}
\end{align}
Now, by Lemma \ref{expdecann}, there exists $\delta>0$ and $k_0$ such that for any $M\in[M_0,M_0^2],$
\begin{equation}\label{init}
\PP\left[\frac{T(A_{N_{M,k_0-1}})}{N_{M,k_0-1}}< \delta\right]\leq C_1 e^{-N_{M,k_0-1}^{\alpha/5}},
\end{equation}
Therefore, up to increasing $k_0$, we have for all $M\in[M_0,M_0^2]$,
\begin{align}\label{usebootstrap2}
\begin{split}
&\PP\left[\frac{T(A_{N_{M,k_0}})}{N_{M,k_0}}< \frac{\delta}{1+\frac{1}{\sqrt{N_{M,k_0-1}}}}\right]
\\
& \quad \leq \left(c_d N_{M,k_0}^{d-1}\sqrt{N_{M,k_0-1}}\right)^{\sqrt{N_{M,k_0-1}}}
\left( (C_1 e^{-N_{M,k_0-1}^{\alpha/5}})^{\sqrt{N_{M,k_0-1}}}
+\sqrt{N_{M,k_0-1}} N_{M,k_0}e^{-\frac12N_{M,k_0-1}^{\alpha/2}}
\right)
\\
&\quad \leq C_1 e^{-N_{M,k_0}^{\alpha/5}}.
\end{split}
\end{align}
Indeed, we have for all $M\in[M_0,M_0^2]$ and $k$ large enough
 $$\sqrt{N_{M,k-1}}-N_{M,k-1}^{\alpha/5}\sqrt{N_{M,k-1}}=N_{M,k-1}^{1/2}-N_{M,k-1}^{\alpha/5+1/2}<-N_{M,k}^{\alpha/5}, $$
and recalling that $\alpha>1$,
$$
\sqrt{N_{M,k-1}}-\frac12N_{M,k-1}^{\alpha/2}=\sqrt{N_{M,k-1}}-\frac12N_{M,k}^{\alpha/4}=N_{M,k}^{1/4}-\frac12N_{M,k}^{\alpha/4} < -N_{M,k}^{\alpha/5}.
$$
so that using (\ref{usebootstrap}) and repeating the reasoning of (\ref{usebootstrap2}), we get by induction that for any $k\geq k_0$, for any $M\in[M_0,M_0^2]$, 
$$
\PP\left[\frac{T(A_{N_{M,k}})}{N_{M,k}}< \delta_\infty\right]\leq C_1e^{-N_{M,k}^{\alpha/5}},
$$
where 
$$\delta_\infty:=\delta\prod\limits_{k=1}^\infty\frac{1}{1+\frac{1}{\sqrt{N_{M_0,k-1}}}}>0,$$
which is the conclusion of Proposition \ref{numberballs}.
\end{proof}

\section{Appendix}

\subsection{Cameron-Martin space}\label{p:CM}
Given a Gaussian field $f$, we introduce a Hilbert space $H$. It is constituted of elements of $\mathcal{C}(\RR^d)$, and called the \emph{Cameron-Martin space} of $f$. To define it, first define the Hilbert space $G$ to be the space of Gaussian random variables of the form
$$
\sum\limits_{i\in\NN}a_if(x_i),
$$
where the $a_i$ are in $\RR$ and the $x_i$ in $\RR^d$, and they satisfy
$$
\sum\limits_{i,j}a_ia_j\kappa(x_i,x_j)<\infty,
$$
$\kappa$ being the covariance kernel of $f$.
We further define the map $P$ from $G$ to $\mathcal{C}(\RR^d)$ by
$$
\xi\mapsto P(\xi)(.):=\EE[\xi f(.)].
$$
\begin{defn}[Cameron-Martin space]\label{d:cm}
The Cameron-Martin space $H$ of $f$ is then the set $P(G)$ equipped with the scalar product
$$
\langle h_1,h_2\rangle:=\EE[P^{-1}(h_1)P^{-1}(h_2)].
$$
\end{defn}
We now explain a construction which is used to exhibit elements of the Cameron-Martin space whose support contains large balls. Suppose that the field $f$ has a spectral density $\rho^2$. Then the Cameron-Martin space of $f$ can be equivalently described as the space 
$$
\tilde{H}=\mathcal{F}[g\rho], \quad\quad g\in L^2_{\text{sym}}(S),
$$
where $\mathcal{F}$ denotes the Fourier transform, $S$ is the support of $\rho$, $L^2_{\text{sym}}(S)$ is the set of complex Hermitian $L^2$ functions supported on $S$ and the inner product is the one associated with $L^2_{\text{sym}}(S)$.
We then have, for any $h\in H$ such that its Fourier transform $\hat{h}$ is defined,
$$
\|h\|_H^2=\int\limits_{\RR^d}\frac{|\hat{h}|^2(x)}{\rho^2(x)}dx.
$$
Using this description, if the field $f$ verifies Assumption \ref{a:basic}, in particular the spectral density assumption \ref{i:spectr},
one can establish the following:
\begin{equation}\label{normcontrol}
\|h\|_H\leq \frac{(\sup |\hat{h}|)\Vol(\Supp(\hat{h}))}{\inf\limits_{\Supp(\hat{h})}|\rho|}.
\end{equation}

The following are the key propositions used to establish the comparisons between the laws of Gaussian fields with close moving-average kernels.
\begin{prop}[Cameron-Martin theorem, see e.g \cite{janson_1997} Theorems 14.1 and 3.33]
\label{CMcontrol}
Let $f$ be a Gaussian field satisfying Assumption \ref{a:basic}. Let $h$ be an element of its Cameron-Martin space $H$. Let $X=P^{-1}(h)$ (see Definition \ref{d:cm}). Then the Radon-Nikodym derivative of the law of $f+h$ with respect to that of $f$ is $\exp[X-\frac12 \EE[X^2]]$. Otherwise stated, if $A$ is an event, then
$$
|\PP[f\in A]-\PP[f+h\in A]|= \left|\EE_f\left[Q(h)\un_A\right]\right|,
$$
where $$Q(h):=1-\exp[X-\frac12 \EE[X^2]].$$
\end{prop}
We will call $Q(h)$ the \emph{Radon-Nikodym difference} associated to $h$.
\begin{prop}[\cite{muirhead2018sharp}, proof of Proposition 3.6 and Corollary 3.10]
\label{CMcontrol2}
There exists a universal constant $c>0$ such that for any Gaussian field $f$, for any element $h$ of its Cameron-Martin space verifying $\|h\|_H\leq c$, its Radon-Nikodym difference $Q(h)$ verifies:
$$
\EE_f\left[Q(h)^2\right]\leq  \frac{\|h\|_H^2}{\log 2},
$$
$\|\|_H$ being as in Definition \ref{d:cm}.
Further, if 
$$N_0:=\inf \{N\in\NN,\quad \inf\limits_{\mathcal{B}_{\frac1N}} \rho\geq \frac{\rho(0)}{2}\},$$
then for any $N\geq N_0$, there is an element $h$ of the Cameron-Martin space verifying $|h|\geq 1$ on $\mathcal{B}_N$ such that
$$
\|h\|_H\leq \frac{C_0 N}{\int\limits_{\RR^d} q},
$$
where $C_0$ is a universal constant.

\end{prop}
\begin{rem}
The latter estimate follows from equation (\ref{normcontrol}), by considering functions $g\in L^2_{\textup{sym}}(S)$ with support on small annuli, and recalling that $q\star q=\kappa(0,.)$ so that $\int\limits_{\RR^d} q=\rho(0)$.
\end{rem}

\bigskip

\subsection{White-noise and subsets}\label{p:WN}
We state a few facts about convolution with the white-noise on $\RR^d$, see \cite{dewan2021upper}, Appendix A for more details.
\begin{defn}
Let $(\varphi_i)_{i\in\NN}$ be a Hilbertian basis of $L^2(\RR^d)$, and $(Z_i)_{i\in\NN}$ be an i.i.d sequence of centered Gaussian random variables of mean $1$. For any $L^2$ map $q$, define
$$
q\star W:=\lim\limits_{n\to\infty}\sum\limits_{i=1}^n q\star  Z_i \varphi_i,
$$
where the limit is that of convergence in law with respect to $\mathcal{C}^0$ topology on compact sets.
\end{defn}
The limit law in this convergence is independent from the Hilbertain basis we have chosen.
Now, if $(S_i)_{i\in\NN}$ is a family of compact sets intersecting only on their boundaries (which have 0 Lebesgue measure), and covering the whole space $\RR^d$, we can define $q\star (W|_{S_i})$ for any $i$ in the same way, with the $\varphi_j$ being this time elements of a Hilbertian basis of $L^2(S_i)$.
The $q\star (W|_{S_i})$ can thus be taken independent to one another and in that case, it is easy to see that the following holds.
\begin{prop}\label{splitboxes}
We have the following equality in law, for any such family $(S_i)_{i\in\NN}$ of $\RR^d$ of compact sets, for any $q\in L^2$.
\begin{equation}\label{wnsplit}
q\star W = \sum\limits_{i=1}^\infty q\star (W|_{S_i}).
\end{equation}
\end{prop}
%\begin{proof}
%For any $i$, we let $(\varphi_{i,j})_{j\in\NN}$ be the basis of $L^2(S_i)$ considered. Let $(Z_{i,j})$ be an i.i.d family of centered Gaussian random variables of mean $1$. Then $(\varphi_{i,j})_{i,j}$ constitutes an orthonormal basis of $L^2(\RR^d)$ and
%\begin{align*}
%\sum\limits_{i=1}^\infty q\star (W|_{S_i})
% = \sum\limits_{i=1}^\infty  \sum\limits_{j=1}^\infty q\star  Z_{i,j} \varphi_{i,j}
% = q\star W.
%\end{align*}

%\end{proof}
\bibliography{mybib}{}
\bibliographystyle{amsplain}
\end{document}